\documentclass[a4paper,12pt]{article}

\usepackage{setspace} 
\usepackage{amsmath}
\usepackage{amssymb}
\usepackage{ascmac}
\usepackage{amsthm}
\usepackage{amscd}
\usepackage{color}
\usepackage{enumerate}
\usepackage[dvips]{graphicx}
\usepackage{float}
\usepackage{wrapfig}

  \makeatletter
    
    \@addtoreset{equation}{section}
  \makeatother

\theoremstyle{definition}
\newtheorem{definition}{Definition}[section]
\theoremstyle{plain}
\newtheorem{theorem}[definition]{Theorem}
\newtheorem{corollary}[definition]{Corollary}
\newtheorem{lemma}[definition]{Lemma}
\newtheorem{proposition}[definition]{Proposition}
\theoremstyle{remark}
\newtheorem{remark}[definition]{Remark}
\newcommand{\Do}{\partial\!\!\!/}

\begin{document}

\title{ \(\mathfrak{sl}(n,\mathbf{H})\)-Current Algebras on \(S^3\) }
\author{Tosiaki Kori 
\\Department of Mathematics\\
Graduate School 
of Science and Engineering\\
Waseda University,\\Tokyo 169-8555, Japan
\\email{ kori@waseda.jp}}
\date{}
\maketitle

\begin{abstract}
	A current algebra is a certain central extension of the Lie algebra of smooth mappings of a given manifold into a finite dimensional Lie algebra.     A loop algebra or an affine Kac-Moody algebra is the simplest example where the manifold is \(S^1\).    In this paper we investigate a central extension of the Lie algebra that is generated by the \(\mathfrak{sl}(n,\mathbf{C})\)-valued  Laurent polynomial type harmonic spinors over \(S^3\).       We introduce a triple of non-trivial 2-cocycles on the space of half spinors ( regarded as quaternion-valued functions ) over \(S^3\) with the aid of the basic vector fields  on \(S^3\), and extend them to 2-cocycles on the Lie algebra 
of smooth mappings from \(S^3\) to \(\mathfrak{gl}(n,\mathbf{H})\).   Then we take its central extension.      The Euler  vector field acts on this Lie algebra as an outer derivation, so 
 we obtain a second extension.   This is what we call the \(\mathfrak{gl}(n,\mathbf{H}))\)-{\it current algebra}.      As a submodule it contains the set of 
\(\mathfrak{sl}(n,\mathbf{C})\)-valued  Laurent polynomial type harmonic spinors on \(S^3\).   
  The Lie subalgebra  thus generated is the \(\mathfrak{sl}(n,\mathbf{H})\)-{\it current algebra}.
   The root space decomposition of this \(\mathfrak{sl}(n,\mathbf{H})\)-current algebra 
 is obtained, and the Chevalley generators are given.  
 \end{abstract}


\medskip

\section{Introduction}

A current algebra is a certain central extension of the Lie algebra of smooth mappings of a given manifold into a finite dimensional Lie algebra \cite{K-W, M}.     A loop algebra is the simplest example where the manifold is \(S^1\).     Loop algebras appear in the simplified model of quantum field theory where the space is one-dimensional and many important facts in the representation theory of loop algebras  were first discovered by physicists.        A loop algebra is also called  an affine Lie algebra and the highest weight theory of finite dimensional Lie algebra was extended to affine Lie algebras, \cite{C, K, M, P-S, W}.      In this paper we shall investigate a  generalization of affine Lie algebras to the Lie algebra of mappings from  \(S^3\)  to a Lie algebra.     As an affine Lie algebra is a central extension of the smooth mappings from \(S^1\) to the complexification of a simple Lie algebra, so our objective is a central extension of the Lie algebra of smooth mappings from \(S^3\) to the quaternification of a simple Lie algebra.   

First we give a brief review of the central extension of the loop algebra after Kac's book \cite{K}.   Let \(L=\mathbf{C}[t,t^{-1}]\) be the Laurent polynomials in \(t\).   We define a \(\mathbf{C}\)-valued function \(c_0\) on \(L\times L\) by 
\[c_0(P,Q)=\frac{1}{2\pi}\int_{S^1}\,\frac{d P}{dt}(t)\cdot Q(t)\,dt.\]
\(c_0\) satisfies the 2-cocycle properties.   
Let \(\mathfrak{g}\) be a simple finite dimensional Lie algebra.   We consider the loop algebra \(L\mathfrak{g}=L\otimes_{\mathbf{C}}\mathfrak{g}\).    This is an infinite dimensional complex Lie algebra with the bracket
\[[P\otimes x,\,Q\otimes y]=PQ\otimes\,[x,y]\,,\quad P,Q\in L,\,x,y\in\mathfrak{g}.\]
Let \((\cdot\vert\cdot)\) be a non-degenerate symmetric bilinear  \(\mathbf{C}\)-valued form on \( \mathfrak{g}\).   We define a \(\mathbf{C}\)-valued 2-cocycle on the Lie algebra \(L\mathfrak{g} \) by 
\[c(P\otimes x,\,Q\otimes y)\,=\, (x\vert y)c_0(P,Q)\,.\]
Let \(\widehat{L\mathfrak{g}} =L\mathfrak{g} \oplus \mathbf{C}a\) (direct sum of vector spaces) be the extension associated to the cocycle \(c\).   The Euler derivation \(t\frac{d}{dt}\) acts on \(\widehat{L\mathfrak{g}} \) as an outer derivation and kills \(c\).    Then  \(\widehat{\mathfrak{g}}\) is  the Lie algebra that is obtained by adjoining a derivation \(d\) to \(\widehat {L\mathfrak{g}} \).    \(\widehat{\mathfrak{g}}\) is a complex vector space 
\[\widehat{\mathfrak{g}}=L\mathfrak{g} \oplus \mathbf{C}a\oplus \mathbf{C}d\]
with the bracket defined as follows \((x,y\in \mathfrak{g},\,\lambda,\nu,\lambda_1,\nu_1\in \mathbf{C})\):
\begin{eqnarray*}
[t^k\otimes x\oplus\lambda c\oplus \nu d, \,t^l\otimes y\oplus\lambda_1c\oplus \nu_1 d\,]
=
(t^{k+l}\otimes[x,y]+\nu lt^l\otimes y-\nu_1 kt^k\otimes x)\oplus k\delta_{k,-l}(x\vert y)c\,.
\end{eqnarray*}
In brief an affine Lie algebra is a central extension of a simple Lie algebra with the Laurent polynomial coefficients.

To develop an analogous theory for the current algebra on \(S^3\) we  consider  the algebra of  {\it Laurent polynomial type harmonic spinors} on \(S^3\).   Then we introduce a triple of  2-cocycles on this algebra.   For this purpose we prepare 
 a rather long introduction to our previous results on analysis of quaternion valued functions ( spinors ) on \(\mathbf{R}^4\) that were developed in \cite{ F,G-M, Ko1, Ko2} and \cite{K-I}, since these subjects seem not to be familiar.    
   
 The space of spinors \(\Delta=\mathbf{C}^2\otimes_{\mathbf{C}}\mathbf{C}^2\) gives an irreducible representation of the complexification of the Clifford algebra;  
\( {\rm Clif }^c_4={\rm Clif }_4\otimes_{\mathbf{C}}\mathbf{C}\).    The Dirac matrices 
\begin{equation*}
\gamma_k\,=\,\left(\begin{array}{cc}0&-i\sigma_k\\ i\sigma_k&0\end{array}\right)\,,\quad k=1,2,3,\quad
\gamma_4\,=\,\left(\begin{array}{cc}0&-I\\ -I&0\end{array}\right)\,,
\end{equation*}
where \(\sigma_k\)'s are Pauli matrices, give the generators of \(\,{\rm Clif }^c_4\,\simeq\, {\rm End} (\Delta)\).
Let 
\(S= \mathbf{C} ^2\times \Delta\)  be the spinor bundle and let \(S^{\pm}\) be the half spinor bundles .      The Dirac operator is given by  
\begin{equation*}
\mathcal{D}=\,-\,\frac{\partial}{\partial x_1}\gamma_4\,-\,\frac{\partial}{\partial x_2}\gamma_3\,-\,\frac{\partial}{\partial x_3}\gamma_2\,-\,\frac{\partial}{\partial x_4}\gamma_1\,:\,C^{\infty}(M, S)\longrightarrow\,
C^{\infty}(M, S)\,.
\end{equation*}  
The half spinor Dirac operator  $D=\mathcal D\vert S^+$  has the polar decomposition:
\begin{equation*}
D = \gamma_+ \left( \frac{\partial}{\partial n} - \Do \right) ,
\end{equation*}
with the tangential (nonchiral) Dirac operator $\Do\,$ on \(S^3\subset\mathbf{C}^2\):
\begin{equation*}
\Do = - \left[ \sum^{3} _{i = 1} \left( \frac{1}{|z|} \theta_i \right) \cdot \nabla_{\frac{1}{|z|} \theta_i} \right]
= \frac{1}{|z|} 
\begin{pmatrix}
-\frac{1}{2} \theta & \,e_+ \\[0.2cm]
-e_- &\, \frac{1}{2} \theta
\end{pmatrix}.
\end{equation*}
In the above  \( \gamma_+: S^+  \longrightarrow S^-  \) is the Clifford multiplication by the radial vector $\frac{\partial}{\partial n}\,\), and  \(\theta,\,e_+\,\) and \(e_-\,\) are the basic vector fields on \(\{|z|= 1\}\simeq S^3\,\).   

The tangential Dirac operator \(\Do \) on \(S^3\) is a self adjoint elliptic differential operator.     
The eigenvalues of \(\Do\) are  \(\{\frac{m}{2},\,\,-\frac{m+3}{2}\,;\,m=0,1,\cdots \} \) with multiplicity \((m+1)(m+2)\),     and we have an explicitly written polynomial formula for eigenspinors \(\left\{ \phi^{+(m,l,k)},\,\phi^{-(m,l,k)}\right\}_{0\leq l\leq m,\,0\leq k\leq m+1}\)  corresponding to the eigenvalue \(\frac{m}{2}\) and  \(-\frac{m+3}{2}\,\) respectively that give rise to a complete orthonormal system in \(L^2(S^3, S^+)\), \cite{Ko1, Ko2}.     A spinor \(\phi\) on a domain \(G\subset \mathbf{C}^2\) is called a {\it harmonic spinor} on \(G\) if \(D\phi=0\).
Each \(\phi^{+(m,l,k)}\) is extended to a harmonic spinor on \(\mathbf{C}^2\), while each \(\phi^{-(m,l,k)}\) is extended to a harmonic spinor on \(\mathbf{C}^2\setminus \{0\}\).   Every harmonic spinor \(\varphi\) on \(\mathbf{C}^2\setminus \{0\}\) has a Laurent series expansion ( or a Fourier series expansion when restricted to \(S^3\) ) by the basis \(\phi^{\pm(m,l,k)}\):
 \begin{equation*}
 \varphi(z)=\sum_{m,l,k}\,C_{+(m,l,k)} \phi^{+(m,l,k)}(z)+\sum_{m,l,k}\,C_{-(m,l,k)}\phi^{-(m,l,k)}(z).
 \end{equation*}
   The set of spinors of Laurent polynomial type is denoted by   \(\mathbf{C}[\phi ^{\pm}] \).   
In \cite{K-I} we proved that the restriction of 
 \( \mathbf{C}[\phi^{\pm}\,]\) to \(S^3\) becomes an associative subalgebra of \(S^3\mathbf{H}=Map(S^3,\mathbf{H})\).   We must note that  \( \mathbf{C}[\,\phi^{\pm}\,]\) itself is not an algebra.      

There is an identification of \(\mathbf{H}\) with \(\mathbf{C}^2\) as \(\mathbf{C}\)-vector spaces:
\(\mathbf{H}\ni x=z_1+jz_2\,\longleftrightarrow\,\left(\begin{array}{c} z_1\\z_2\end{array}\right)\in \mathbf{C}^2\,\), 
that yields an identification of the \(\mathbf{H}\)-valued functions \(\,S^3\mathbf{H}\,\) and the even half spinors  \(\,C^{\infty}(S^3, S^+)\); 
\( S^3\mathbf{H}\ni u+jv\longleftrightarrow\,\phi=\begin{pmatrix} u\\ v \end{pmatrix}\in C^{\infty}(S^3, S^+)\,\).   

We have the following \(\mathbf{R}\)-bilinear bracket on  \(S^3\mathbf{H}\):
\begin{equation}\label{1Liebrintr}
\bigl[\,u _1+jv_1, \,u _2+jv_2\,\bigr]\,=\, (v_1 \Bar{v} _2 - \Bar{v}_1 v_2)+\,j((u _ 2 - \Bar {u} _2 ) v _1 - (u _1 - \Bar {u} _1) v_2)\,.\end{equation}
 Correspondingly,  
 \begin{equation}\label{Liebrintr}
\bigl[\,\phi _1\, , \,\phi _2\,\bigr]=
 \begin{pmatrix} \,v_1 \Bar{v} _2 - \Bar{v}_1 v_2 \,\\[0.2cm]
\,(u _ 2 - \Bar {u} _2 ) v _1 - (u _1 - \Bar {u} _1) v_2\, \end{pmatrix} \,,
\end{equation}   
for even half spinors   $\phi _1 = \begin{pmatrix} u_1\\ v_1 \end{pmatrix}$ and $\phi _2 = \begin{pmatrix} u_2\\ v_2 \end{pmatrix}$.
These brackets are \(\mathbf{R}\)-bilinear antisymmetric form and satisfies the Jacobi identity.

Hitherto we prepared the space of basic fields  \(S^3\mathbf{H}\simeq C^{\infty}(S^3,\,S^+)\) and \(\mathbf{C}[\phi^{\pm}]\) that play the role of coefficients of current algebras.   They are \(\mathbf{H}\)-modules or \(\mathbf{C}\)-modules endowed with real Lie brackets.  
 To convince the readers that these objects prepare a good background to develop our subjects we shall introduce  
 {\it quaternion Lie algebras}.     
 
 A {\it quaternionic structure} on  a  \(\mathbf{C}\)-module \(V\) is a conjugate linear map \(J:\,V\mapsto V\)  that  satisfies  the relation \(J^2\,=\,-\,I\,\), \cite{A}.   This is equivalent to the fact that \(V\) is a \(\mathbf{H}\)-module \(V=\mathbf{H}\otimes_{\mathbf{C}} V_0=V_0+JV_0\) for a 
 \(\mathbf{C}\)-submodule \(V_0\).   There is an involution \(\sigma\) on \(V\) defined by 
\(\sigma(x+Jy)=x-Jy\), \(x,y\in V\).   Then  \(V_0\), respectively \(JV_0\), is the eigenspace of \(\sigma\) corresponding to the eigenvalue \(+1\), respectively \(-1\).     A typical example is  \(\mathfrak{gl}(n,\mathbf{H})=\mathbf{H}\otimes_{\mathbf{C}} \mathfrak{gl}(n,\mathbf{C})=\mathfrak{gl}(n,\mathbf{C})+J\mathfrak{gl}(n,\mathbf{C})\).    Let \(V=V_0+JV_0\) be a \(\mathbf{H}\)-module.   We call a \(\mathbf{C}\)-module \(W\) a {\it semi-submodule} of \(V\) if \(W\) is invariant under \(\sigma\), or equvalently, \(W\) respects the \(\mathbf{Z}_2\)-gradation; \(W=W\cap V_0+W\cap JV_0\).      \(W\) is not necessarily \(J\)-invariant.   For example 
\(\mathfrak{sl}(n,\mathbf{H})\) is a semi-submodule of \(\mathfrak{gl}(n,\mathbf{H})\) which is not a \(\mathbf{H}\)-submodule.    While \(\mathfrak{sl}(n,\mathbf{C})+J\mathfrak{sl}(n,\mathbf{C})\) is \(J\)-invariant and becomes a \(\mathbf{H}\)-submodule of \(\mathfrak{gl}(n,\mathbf{H})\).    
        
The definition of a {\it quaternion Lie algebra} is as follows:
 \begin{definition}\label{ql}
 \begin{enumerate}
 \item
Let  \(\mathfrak{g}\) be a semi-submodule of a quaternion module \(V=V_0+JV_0\).   The vector space \(\mathfrak{g}\) endowed with an operation \(\mathfrak{g}\times \mathfrak{g}\longrightarrow \mathfrak{g}\), denoted by the bracket: \((X,Y)\longrightarrow [X,Y]\) for any \(X,Y\in\mathfrak{g}\), is called a {\it quaternion Lie algebra } if the following properties are satisfied:
 \begin{enumerate}
 \item
 The bracket operation is \(\mathbf{R}\)-bilinear.
 \item
 \[
 [\,X\,,\,Y\,]\,+\, [\,Y\,,\,X\,] =0 \qquad\mbox{ for all \(X,Y\in \mathfrak{g}\)}.\]
 \item
 \[[\,X\,,\,[\,Y,\,Z]\,]\,+\,[\,Y\,,\,[\,Z,\,X]\,]\,+\,[\,Z\,,\,[\,X,\,Y]\,]\,=\,0\qquad \forall X,Y,Z\in \mathfrak{g}.\]
 \item
 \[\sigma [\,X\,,\,Y\,]\,=\,[\,\sigma X\,,\,\sigma Y\,\,] \qquad\mbox{ for all \(X,Y\in \mathfrak{g}\)}.\]
 \end{enumerate}
 \item
 Let  \(\mathfrak{g}_0\) be a complex Lie algebra.   Let  \(\mathfrak{g}\subset V_0+JV_0\) be a 
 quaternion Lie algebra such that \(\mathfrak{g}_0=\mathfrak{g}\cap V_0\).     If the Lie algebra structure on  \(\mathfrak{g}\) restricts  to    
\(\mathfrak{g}_0\)   
we call  \(\mathfrak{g}\)  {\it the quaternification of} \(\,\mathfrak{g}_0\).
\end{enumerate}
\end{definition}

\(\mathfrak{gl}(n,\mathbf{H})= \mathbf{H}\otimes_{\mathbf{C}}\mathfrak{gl}(n,\mathbf{C})\) is the quaternification of \(\mathfrak{gl}(n,\mathbf{C})\).           We shall often identify \(\mathfrak{gl}(n,\mathbf{H})\) with \(MJ(2n,\mathbf{C})=\{ Z\in \mathfrak{gl}(2n,\mathbf{C});\,JZ=\overline ZJ\} \), see (\ref{mj2nc}).    These are isomorphic as \(\mathbf{C}\)-modules.   It is easy to verify that \(MJ(2n,\mathbf{C})\)  is a quaternion Lie algebra. The involution \(\sigma\) on \(MJ(2n, \mathbf{C})\) is the corresponding matrix notation of \(A+JB\longrightarrow A-JB\) of  \(\mathfrak{gl}(n,\mathbf{H})\).

    In section 2 we shall investigate the quaternion Lie algebra \(S^3\mathfrak{gl}(n,\mathbf{H})=Map(S^3,\mathfrak{gl}(n,\mathbf{H}))\), and construct central extensions of \(S^3\mathfrak{gl}(n,\mathbf{H})\).
  
  For a  
  \(\mathbf{C}\)-vector space \(V\) with a real Lie algebra bracket \(\,[\,\cdot\,,\,\cdot\,]_V\,\),  
 a {\it central extension} of   \((V,\,[\,\cdot\,,\,\cdot\,]_V\,)\) consists of a \(\mathbf{C}\)-vector space 
\(W=V\oplus Z\) ( direct sum ) with a real linear bracket \(\,[\,\cdot\,,\,\cdot\,]_W\) over \(W\) such that 
\[\,Z\,\subset Z(W)=\{w\in W\,:\,[w,W]_W=0\,\}\,,\]
and such that  \(\,[\,\cdot\,,\,\cdot\,]_W\)  restricts to  \(\,[\,\cdot\,,\,\cdot\,]_V\).   
 
  Let \(\mathfrak{gl}(n,\mathbf{H})\) be the algebra of \(n\times n\)-matrices with entries in \(\mathbf{H}\).     
Then the \(\mathbf{C}\)-vector space \(S^3\mathfrak{gl}(n,\mathbf{H})=S^3\mathbf{H}\otimes_{\mathbf{C}}  \mathfrak{gl}(n,\mathbf{C})\)  is equiped with a real Lie algebra structure given by:
  \[
 [\,\phi \otimes E_{ij}\, , \,\psi  \otimes E_{kl}\,]_{ S^3\mathfrak{gl}(n,\mathbf{H})} =  ( \phi\cdot\psi ) \otimes \delta_{jk}E_{il} \,
- (\psi \cdot \phi ) \otimes \delta_{il}E_{kj}
 , 
\]
where \(\phi,\,\psi\,\in S^3\mathbf{H}\, \), and  \(E_{ij}\)  is the \(n\times n\)-matrix with entry \(1\) at \((i,j)\)-place and \(0\) otherwise.    
 \(S^3\mathfrak{gl}(n,\mathbf{H})\) becomes a quaternion Lie algebra and it is the quaternification of  \(S^3\mathfrak{gl}(n,\mathbf{C})\):
   \[S^3\mathfrak{gl}(n,\mathbf{H})=\mathbf{H}\otimes_\mathbf{C}S^3\mathfrak{gl}(n,\mathbf{C})=\,S^3\mathfrak{gl}(n,\mathbf{C})+J\,S^3\mathfrak{gl}(n,\mathbf{C})\,.\]
      The involution \(\sigma\) is given by
 \(\sigma (\phi\otimes X)=(\sigma\phi)\otimes X\,\).   

We proceed to the central extension of the quaternion Lie algebra \(S^3\mathfrak{gl}(n,\mathbf{H})\).   
Our 2-cocycles on \(C^{\infty}(S^3, S^+)\) are defined as follows.   We put 
\[\Theta_k\phi=\,\frac{1}{2}\,\left(\begin{array}{c}\,\theta_k\, u\\[0.3cm] \,\theta_k\, v\end{array}\right),\, k=0,1,2,\] 
for \(\phi = \begin{pmatrix} u\\ v \end{pmatrix}\).    Where  
\(\theta_0=\sqrt{-1}\theta,\quad \theta_1=e_++e_-,\quad \theta_2=\sqrt{-1}(e_+-e_-)\,\).    We introduce the following three non-trivial real valued 2-cocycles \(c_k,\,k=0,1,2\), on  \(\,S^3\mathbf{H}\simeq C^{\infty}(S^3, S^+)\) : 
\begin{equation}\label{2cocycleintr}
c_k(\phi_1,\phi_2)\,=\,
\,\frac{1}{2\pi^2}\int_{S^3}\,\,tr\,(\,\Theta_k \phi_1\cdot \phi_2\,)\, d\sigma,\qquad \forall \phi_1\,,\, \phi_2\in\,C^{\infty}(S^3, S^+) \,.   
  \end{equation} 

Each 2-cocycle \(c_k\), \(k=0,1,2\), of  (\ref{2cocycleintr}) on \(S^3\mathbf{H}\,\) is extended to \(S^3\mathfrak{gl}(n,\mathbf{H})\) by the formula 
\[\tilde{c}_k(\phi\otimes X,\,\psi\otimes Y)\,=\,(X \vert Y)\,c_k(\phi,\psi),\quad \phi,\,\psi\in S^3\mathbf{H},\,\,X,Y\in \mathfrak{gl}(n,\mathbf{C}),\]
where \((X\vert Y)=Trace(XY)\) is the Killing form of \(\mathfrak{gl}(n,\mathbf{H})\).   
Then we have the associated central extension:
 \[\,
 S^3\mathfrak{gl}(n,\mathbf{H})\oplus(\oplus_{k=0}^2\mathbf{C}a_k) ,\]
 which is a quaternion Lie algebra.

 As a Lie subalgebra of \(S^3\mathfrak{gl}(n,\mathbf{H})\) we have the Lie algebra \(\mathbf{C}[\phi^{\pm}]\otimes_{\mathbf{C}}   \mathfrak{gl}(n,\mathbf{C})\) of \(\mathfrak{gl}(n,\mathbf{C})\)-valued Laurent polynomial spinors on \(S^3\,\).   \(\mathbf{C}[\phi^{\pm}]\otimes_{\mathbf{C}}   \mathfrak{gl}(n,\mathbf{C})\) is also a quaternion Lie algebra.    
 We denote  it by \(\,\widehat{\mathfrak{gl}(n,\mathbf{H})}=\mathbf{C}[\phi^{\pm}]\otimes_{\mathbf{C}}   \mathfrak{gl}(n,\mathbf{C})\).  
  \(\,\widehat{\mathfrak{gl}(n,\mathbf{H})}\) has the central extension by the 2-cocycles \(\tilde{c}_k\), \(k=0,1,2\), as well:   
   \[\widehat{\mathfrak{gl}(n,\mathbf{H})}(a)=  \mathbf{C}[\phi^{\pm}]\otimes_{\mathbf{C}}   \mathfrak{gl}(n,\mathbf{C})\oplus(\oplus_{k=0}^2\,\mathbf{C}a_k).\]
      The radial vector field \(\mathbf{n}_0=\frac{\partial}{\partial n}\)  is extended to act on \(\,\widehat{\mathfrak{gl}(n,\mathbf{H})}\) as an outer derivation.     Then we have the second extension: 
\begin{equation*}
\widehat{\mathfrak{gl}}=\mathbf{C}[\phi^{\pm}]\otimes_{\mathbf{C}} \mathfrak{gl}(n,\mathbf{C})\oplus(\oplus_{k=0}^2\mathbf{C}a_k)\oplus\mathbf{C}\mathbf{n}.
\end{equation*}

The Lie bracket of the quaternion Lie algebra  $\widehat{\mathfrak{gl}}\) is given by 
 \begin{eqnarray*}
 [\,\phi \otimes X\, , \,\psi  \otimes Y\,]_{\widehat{\mathfrak{gl}}} 
  &=&  (\phi\cdot\psi)\otimes\,(XY) - (\psi \cdot \phi) \otimes (YX)+ \, (X\vert Y) \sum_{k=0}^2\tilde{c}_k(\phi , \psi )\,a_k \, , \nonumber
\\[0,3cm] 
 [\,a_k\,, \,\phi\otimes X\,] _{\widehat{\mathfrak{gl}}}&=&0\,, \qquad  [\,\mathbf{n}\,,\,a_k\,]_{\widehat{\mathfrak{gl}}}\,=0, \quad k=0,1,2,\\[0,3cm]
  [\,\mathbf{n}, \,\phi \otimes  X\,] _{\widehat{\mathfrak{gl}}}&=&\,\mathbf{n}_0\phi \otimes  X\, .
  \end{eqnarray*}
for \(\phi,\,\psi\,\in \,\mathbf{C}[\phi^{\pm}]\, \) and any base \(X,\,Y\) of \(\mathfrak{gl}(n,\mathbf{C})\).

In section 3 we construc a quaternification of \(Map(S^3,\,\mathfrak{sl}(n,\mathbf{C}))\) and its central extensions.        Let \(\mathfrak{sl}(n,\mathbf{H})\) denote the quaternion special linear algebra, i.e. the algebra of   \(n\times n\)-matrices of entries in \(\mathbf{H}\) with zero traces.   It is in fact a quaternion Lie algebra in the sense of  our Definition \ref{ql}.         
 Then  
 \(S^3\mathfrak{sl}(n,\mathbf{H})\) is a quaternion Lie subalgebra of   \(S^3\mathfrak{gl}(n,\mathbf{H})\).
We see that 
\(\mathbf{C}[\phi^{\pm}]\otimes_{\mathbf{C}} \mathfrak{sl}(n,\mathbf{C})\) is a semi-submodule of the quaternion Lie algebra 
\(\widehat{\mathfrak{gl}(n,\mathbf{H})}=\mathbf{C}[\phi^{\pm}]\otimes_{\mathbf{C}}\mathfrak{gl}(n,\mathbf{C})\).      Then we define {\it the \(\mathfrak{sl}(n,\mathbf{H})\)-current algebra} as the Lie subalgebra of \(
\widehat{\mathfrak{gl}(n,\mathbf{H})}\,\) generated by \(\mathbf{C}[\phi^{\pm}]\otimes_{\mathbf{C}} \mathfrak{sl}(n,\mathbf{C})\), and denote it by \(\widehat{\mathfrak{sl}(n,\mathbf{H})}\).

By the 2-cocycles \(\tilde{c}_k\), \(k=0,1,2\), we have the central extension
 \begin{equation*}
 \,\widehat{\mathfrak{sl}(n,\mathbf{H})}(a)=\widehat{\mathfrak{sl}(n,\mathbf{H})}\oplus (\oplus_{k=0}^2\mathbf{C}a_k)\,.\
 \end{equation*}
 Further  we have the extension of \( \widehat{\mathfrak{sl}(n,\mathbf{H})}(a)\) by the outer derivation \(\mathbf{n}_0\): 
\begin{equation*}
\widehat{\mathfrak{sl}}\,=\, \widehat{\mathfrak{sl}(n,\mathbf{H})}\oplus (\oplus_{k=0}^2\mathbf{C}a_k)\oplus (\mathbf{C}\mathbf{n})\,.
\end{equation*}
    These are Lie subalgebras of  \(\,\widehat{\mathfrak{gl}}=\widehat{\mathfrak{gl}(n,\mathbf{H})} \oplus  (\oplus_{k=0}^2\mathbf{C}a_k)\oplus  ( \mathbf{C}\,\mathbf{n})\,\). 

Finally we discuss the root space decomposition of \(\,\widehat{\mathfrak{sl}}\,\).     Let \(\mathfrak{h}\) be a  Cartan subalgebra of \(\mathfrak{sl}(n,\mathbf{C})\).    Let \(\Delta\) be the 
set of roots and let \(\Pi\) be the set of simple roots of \(\mathfrak{sl}(n,\mathbf{C})\).   For a root \(\alpha\in\Delta\),  \(\mathfrak{g}_{\alpha}\) denotes the root space of \(\alpha\).   
Let 
\begin{equation*}
\widehat{\mathfrak{h}}\,=\,
 (\,(\, \mathbf{C}\,\phi ^{+(0, 0,1)}\,) \otimes\mathfrak{h} )\,\oplus  (\oplus_{k=0}^2\mathbf{C}a_k)\oplus (\mathbf{C}\,\mathbf{n} )\,
 =
 \mathfrak{h}\oplus  (\oplus_{k=0}^2\mathbf{C}a_k)\oplus (\mathbf{C} \mathbf{n})\,.
\end{equation*}
  \(\widehat{\mathfrak{h}}\) is a commutative subalgebra of \(\widehat{\mathfrak{sl}}\,\) and 
 \(\,\widehat{\mathfrak{sl}}\) is decomposed into a direct sum of the simultaneous eigenspaces of \(ad\,(\hat h)\), \(\,\hat h\in \widehat{\mathfrak{h}}\,\).      
 $\Delta \subset \mathfrak{h}^*$ is regarded as a subset of $\,\widehat{\mathfrak{h}}^{\,*}$.    
We introduce  \(\nu\,\),\(\,\Lambda_k \in \widehat{\mathfrak{h}}^{\,*}$ for \(k=0,1,2\), as the dual elements of \(\mathbf{n}\) and \(a_k\), \(k=0,1,2\), respectively, so that 
  the set  \(\{\,\Pi\,, \,\Lambda_0 ,\,\Lambda_1 ,\,\Lambda_2,\,\nu \,\} \)  completes a basis of  $\widehat{\mathfrak{h}}^{\,*}$.       The set of roots of  the representation \(\left(\widehat{\mathfrak{sl}}\,,ad(\widehat{\mathfrak{h}})\right)\) is  
\begin{equation*}
\widehat{\Delta} = \left\{ \frac{m}{2} \nu + \alpha\,;\quad \alpha \in \Delta\,,\,m\in\mathbf{Z}\,\right\} \bigcup \left\{ \frac{m}{2} \nu ;\quad  m\in \mathbf{Z} \, \right\}  \,.
\end{equation*}
\(\widehat{ \mathfrak{sl}}\) has the weight space decomposition:
\begin{equation*}
\widehat{ \mathfrak{sl}}\,=\, \bigoplus_{m\in\mathbf{Z} }\, \widehat{ \mathfrak{g}}_{\frac{m}{2}\nu}\,\bigoplus\,\,\bigoplus_{\alpha\in \Delta,\, m\in\mathbf{Z} }\, 
\widehat{ \mathfrak{g}}_{\frac{m}{2}\nu+\alpha}\,.
\end{equation*}
Each weight space is given as follows.    
 \begin{eqnarray*}  
\widehat{ \mathfrak{g}}_{\frac{m}{2}\nu+ \alpha}\,&=&\mathbf{C}[\phi ^{\pm};\,m\,] \otimes_{\mathbf{R}} \mathfrak{g} _{ \alpha}\,,\quad\mbox{ for \(\alpha\neq 0\) and and \(m\in \mathbf{Z}\)}, \,,\\[0.2cm] 
\widehat{ \mathfrak{g}}_{0\nu}&=& (\,\mathbf{C}[\phi^{\pm};0\,]  \otimes_{\mathbf{R}} \mathfrak{h}\,)\oplus (\oplus_{k=0}^2\mathbf{C}a_k)\oplus(\mathbf{C}\mathbf{n})\,\supset\,\widehat{\mathfrak{h}}\,, 
\\[0.2cm]
 \widehat{ \mathfrak{g}}_{\frac{m}{2}\nu}&=&  \,\mathbf{C}[\phi^{\pm};\,m\,]  \otimes_{\mathbf{R}} \mathfrak{h}\,\,, \quad\mbox{for  \(0\neq  m\in\mathbf{Z} \) . }\,
 \end{eqnarray*}
 Where \(\mathbf{C}[\phi^{\pm};\,m]\) is the subspace of \(\mathbf{C}[\phi^{\pm}]\) constituted of those elements that are of homogeneous degree \(m\):  \(\phi(z)=\vert z\vert^m\phi(\frac{z}{\vert z\vert})\).

 Since a current algebra is the infinitesimal counterpart of a current group we expect to investigate a extensions of  the Lie group of smooth mappings of \(S^3\) to a Lie group.   J. Mickelsson gave an abelian extension of \(Map(S^3,SU(N))\) for \(N\geq 3\) by introducing a 2-cocycle now called after his name,  \cite{M2}.     The associated abelian extension was given by the affine dual of the space \(Map(S^3,su(N))\), \cite{Ko3, M2}.       We do not know the relation of the latter to the present central extension.

\section{Preliminaries on spinor analysis on $S^3\subset \mathbf{C}^2$}

Here we prepare a fairly long preliminary because  I think various subjects belonging to quaternion analysis or detailed properties of harmonic spinors of the Dirac operator on \(\mathbf{R}^4\) are not so familiar to the readers.    We refer to \cite{F,  Ko1, Ko2, K-I}.

\subsection{Quaternions $\mathbf{H}$}

Let  \(\mathbf{H}\) be the quaternion numbers.    A general quaternion is of the form \(\,x=x_1+x_2i+x_3j+x_4k\,\) with 
\(x_1,x_2,x_3,x_4\in \mathbf{R}\).   By taking \(x_3=x_4=0\) the complex numbers \(\mathbf{C}\) are contained in \(\mathbf{H}\) if we identify \(i\) as the usual complex number.    Every quaternion 
\(x\) has a unique expression   
\(x=z_1+jz_2\) with \(z_1,z_2\in\mathbf{C}\).    The quaternion multiplication will be from the right \(x\longrightarrow xy\) :
\begin{equation}\label{rmulti}
xy=( z_1+jz_2\,)(  w_1+jw_2\,)=(z_1w_1-\overline z_2 w_2)+j(\overline z_1w_2+z_2w_1),
\end{equation}
for \(x=z_1+jz_2\), \(y=w_1+jw_2\).    
Especially \(\mathbf{C}\) acts on \(\mathbf{H}\) from the right and 
 \(\mathbf{H}\) and \(\mathbf{C}^2\) are isomorphic as \(\mathbf{C}\)-vector spaces:
 \begin{equation}\label{hccoresp}
 \mathbf{H}\,\stackrel{\sim}{\longrightarrow}\,\mathbf{C}^2\,,\quad z_1+jz_2\longrightarrow
\left(\begin{array}{c}z_1\\z_2\end{array}\right)\,.
\end{equation}
The multiplication  of an element  \(g=a+jb\in \mathbf{H}\) to \(\mathbf{H}\) from the left yields an endomorphism in \(\mathbf{H}\): \(\{x\longrightarrow gx\}\in End_{\mathbf{H}}(\mathbf{H})\).
Under the above identification \(\mathbf{H}\simeq\mathbf{C}^2\)  the left quaternion multiplication  is expressed by a \(\mathbf{C}\)-linear map 
\begin{equation}
\mathbf{C}^2\ni z=\left(\begin{array}{c}z_1\\z_2\end{array}\right)\,\longrightarrow gz=
\,\left(\begin{array}{cc}
a&-\overline b\\[0.2cm] b&\overline a\end{array}\right)\left(\begin{array}{c}z_1\\z_2\end{array}\right)\,\in \mathbf{C}^2\,.\end{equation}
This establishes the \(\mathbf{C}\)-linear isomorphism
\begin{equation}\label{qtomj}
\mathbf{H}\,\ni\, a+jb\,\stackrel{\simeq}{\longrightarrow}\, \left(\begin{array}{cc}
a&-\overline b\\[0.2cm] b&\overline a\end{array}\right)\,\in MJ(2,\mathbf{C})=\left\{\left(\begin{array}{cc}
a&-\overline b\\[0.2cm] b&\overline a\end{array}\right)\,:\quad a,b\in\mathbf{C}\right\},
\end{equation}
that implies also a \(\mathbf{R}\)-algebra isomorphism:
 \begin{equation}
 \mathbf{H}\simeq\,End_{\mathbf{H}}(\mathbf{H})\simeq MJ(2,\mathbf{C})\,.
 \end{equation}

The algebra of quaternion numbers \(\mathbf{H}\) is also a simple example of Clifford algebra.      Here we put a short introduction of Clifford algebras.    The readers can find many excellent text books on this subject, \cite{F, G-M}.
Let \(\mathbf{K}\) be the field \(\mathbf{R}\) or \(\mathbf{C}\).   
Let \(V\) be a \(\mathbf{K}\)-vector space  equipped with a  quadratic form \(q\) over the field \(\mathbf{K}\).   The Clifford algebra 
\(C_{\mathbf{K}}(V,q)\) is a \(\mathbf{K}\)-algebra which contains \(V\) as a sub-vectorspace and is generated by 
the elements of \(V\) subject to the relations
\[v_1v_2+v_2v_1=2q(v_1,v_2)\,,\]
for \(v_1,\,v_2\in V\).    Thus if \(e_1,\cdots,e_n\) is an orthogonal basis of \(V\) then \(C_{\mathbf{K}}(V,q)\) is multiplicatively generated by the elements \(e_1,\cdots,e_n\) that satisfy the relations
\[e_i^2=q(e_i)\,,\quad e_ie_j+e_je_i=0\,,\quad i\neq j\,.\]
The \(2^n\) elements \(e_{i_1}e_{i_2}\cdots e_{i_k}\) for \(i_1<i_2<\cdots<i_k\) form a vector space basis of \(C_{\mathbf{K}}(V,q)\).  

The Clifford algebra \(C_{\mathbf{C}}(V\otimes_{\mathbf{R}}\mathbf{C},\,q_{\mathbf{C}})\) coincides with the complexification \(C_{\mathbf{R}}(V,q)\otimes_{\mathbf{R}}\mathbf{C}\) of the \(\mathbf{R}\)-algebra \(C_{\mathbf{R}}(V,q)\) 
where \(q_{\mathbf{C}}(v_1\otimes a_1,v_2\otimes a_2)=q(v_1,v_2)a_1a_2\).   

  In the sequel we denote 
 \(
{\rm Clif}_n={\rm Clif}_{\mathbf{R}}(\mathbf{R}^n,\,-x_1^2-\cdots-x_n^2)\) and 
\({\rm Clif}_n^c={\rm Clif}_n\otimes_{\mathbf{R}}\mathbf{C}\).   
 \({\rm Clif}_{n}^c\) coincidnes with the Clifford algebra \(C_{\mathbf{C}}(\mathbf{C}^n,z_1^2+\cdots+z_n^2)\).   
We have an important isomorphism:
\begin{equation} 
{\rm Clif}_{n+2}^c={\rm Clif}_n^c\otimes_{\mathbf{C}}\mathbf{C}(2)\,.  
\end{equation}

 Returning to \(\mathbf{H}\) and  \( MJ(2,\mathbf{C})\,\), the matrices of \(MJ(2,\mathbf{C})\) corresponding to \(i,\,j,\,k\in\mathbf{H}\) are 
\begin{equation}\label{3basis}
e_3=
\left(\begin{array}{cc}
i&0\\[0.2cm] 0&-i\end{array}\right)\,,\, e_2=\left(\begin{array}{cc}
0&-1\\[0.2cm] 1&0\end{array}\right)\,,\, e_1=\left(\begin{array}{cc}
0&-i\\[0.2cm] -i&0\end{array}\right)\,.
\end{equation}
We have the relation 
 \begin{equation}\label{relation}
e_3^2=e_1^2=-1\,,\quad e_1e_3+e_3e_1=0\,.
\end{equation}
The relation (\ref{relation}) shows that  \(\{e_1\,,\, e_3\}\)  generate the Clifford algebra 
 \({\rm Clif}_2\), so that \(\mathbf{H}\) and \( MJ(2,\mathbf{C})\) are isomorphic to \({\rm Clif}_2\).   
  
 {\bf examples}

 In the following \(\mathbf{K}(m)\) denote the algebra of \(m\times m\)-matrices with entries in the field \(\mathbf{K}\).
\[ {\rm Clif}_2=MJ(2,\mathbf{C})=\mathbf{H}\,,\quad {\rm Clif}_2^c =\mathbf{C}(2) \,=\mathbf{H}\otimes_{\mathbf{R}}\mathbf{C}\,,\quad  {\rm Clif}_4= \mathbf{H}(2), . \]
\[ {\rm Clif}_{4}^c={\rm Clif}_2^c\otimes_{\mathbf{C}}\,\mathbf{C}(2)=\mathbf{C}(2)\otimes_{\mathbf{C}}\mathbf{C}(2)=\mathbf{C}(4)=\mathbf{H}(2)\otimes_{\mathbf{R}}\mathbf{C}. \]

 For our later convenience we note also the Pauli matrices that represent the left multiplication by \(i,j\) and \(k\) on \( \mathbf{H}=\mathbf{C}^2\):
\begin{equation}\label{pauli}
\sigma_1=\left(\begin{array}{cc}
0&1\\[0.2cm] 1&0\end{array}\right)\,,
\, \sigma_2=\left(\begin{array}{cc}
0&-i\\[0.2cm] i&0\end{array}\right)\,,\, 
\sigma_3=
\left(\begin{array}{cc}
1&0\\[0.2cm] 0&-1\end{array}\right)\,.
\end{equation}
It holds \(e_1=-i\sigma_1\,,\, e_2=-i\sigma_2\,,\,e_3=i\sigma_3\,\).

 \subsection{ Spinors and Dirac operator on \(\mathbf{R}^4\)}
 
The vector space of complex 4-spinors is
\(\Delta=\mathbf{C}^2\otimes_{\mathbf{C}}\mathbf{C}^2\).   
Using this we obtain 
the so-called spinor representation of the Clifford algebra \({\rm Clif }^c_4\,\):
\begin{equation} {\rm Clif }^c_4= {\rm Clif } ( \mathbf{R} ^4)\otimes_{\mathbf{R}}  \mathbf{C}\,\simeq\, {\rm End} (\Delta).
\end{equation}  

We have seen \({\rm Clif}_4^c={\rm Clif}_2^c\otimes _{\mathbf{C}}\mathbf{C}(2)=\mathbf{C}(4)\) and that      \({\rm Clif }_2^c=\,MJ(2,\mathbf{C})\otimes_{\mathbf{R}}\mathbf{C}\) is generated by \(\{\,i\sigma_k\,;\,k=1,2,3\,\}\).   So  \(\,{\rm Clif }_4^c\) is 
generated by  the following Dirac matrices:
\begin{equation}
\gamma_k\,=\,\left(\begin{array}{cc}0&-i\sigma_k\\ i\sigma_k&0\end{array}\right)\,,\quad k=1,2,3,\quad
\gamma_4\,=\,\left(\begin{array}{cc}0&-I\\ -I&0\end{array}\right)\,.
\end{equation}
The set 
\begin{equation}
\left\{\gamma_p, \quad \gamma_p\gamma_q,\quad \gamma_p\gamma_q\gamma_r,\quad\gamma_p\gamma_q\gamma_r\gamma_s\,;\quad 1\leq p,q,r,s\leq 4\,\right\}
\end{equation}
gives a 16-dimensional basis of the representation \({\rm Clif}^c_4\,\simeq\, {\rm End} (\Delta)\,\) with the following relations:
\begin{equation}
\gamma_p\gamma_q+\gamma_q\gamma_p=2\delta_{pq}\,.
\end{equation}
The representation \(\Delta\) decomposes into irreducible representations \(\Delta^{\pm}=\mathbf{C}^2\) of \(\,{\rm Spin}(4)\). 

Let 
\(S= \mathbf{C} ^2\times \Delta\) be the trivial spinor bundle on \( \mathbf{C} ^2\).   The corresponding bundle 
\(S^+= \mathbf{C} ^2\times \Delta^+\) ( resp.  \(S^-= \mathbf{C} ^2\times \Delta^-\) )  is called the even ( resp.  odd ) half spinor bundle and the sections are called even ( resp. odd ) spinors.   
On the other hand, since \({\rm Clif}^c_4=\mathbf{H}(2)\otimes_{\mathbf{R}}\mathbf{C}\) and \(\Delta=\mathbf{H}^2=\mathbf{H}\oplus\mathbf{H}\),   we may look an even spinor on \(M\subset \mathbf{R}^4\) as a \( \mathbf{H}\) valued smooth function:
 \(
C^{\infty}(M, \mathbf{H})\,=\, C^{\infty}(M, S^+)\).   We feel free to use the alternative notation as in 
(\ref{hccoresp}) to write a spinor:
\begin{equation}
C^{\infty}(M, \mathbf{H})\,\ni\,u+jv\,\longleftrightarrow\,\left(\begin{array}{c}u\\v\end{array}\right)\,\in\,C^{\infty}(M, S^+)\,.
\end{equation}

 The Dirac operator is defined by
\begin{equation}
\mathcal{D} = c \circ d\,:\, C^{\infty}(M, S)\,\longrightarrow\,  C^{\infty}(M, S)\,.
\end{equation}
where $d : S \rightarrow  T^{*} \mathbf{C}^2\otimes S \simeq  T \mathbf{C}^2\otimes S $ is the covariant derivative which is the exterior differential in this case, 
and $c: T \mathbf{C}^2 \otimes S \rightarrow S$ is the bundle 
homomorphism coming from the Clifford multiplication.
By means of the decomposition $S = S^{+} \oplus S^{-}$ the Dirac operator has 
the chiral decomposition:
\begin{equation}
\mathcal{D} = 
\begin{pmatrix}
0 & D^{\dagger} \\
D & 0
\end{pmatrix}
: C^{\infty}(\mathbf{C}^2, S^{+} \oplus S^{-}) \rightarrow C^{\infty}(\mathbf{C}^2, S^{+} \oplus S^{-}).
\end{equation}

With respect to the Dirac matrices \(\gamma_{j}\), \(j=1,2,3,4\),  the Dirac operator has the expression:
\begin{equation}
\mathcal{D}=\,-\,\frac{\partial}{\partial x_1}\gamma_4\,-\,\frac{\partial}{\partial x_2}\gamma_3\,-\,\frac{\partial}{\partial x_3}\gamma_2\,-\,\frac{\partial}{\partial x_4}\gamma_1\,.
\end{equation}
If we adopt the notation
\[\frac{\partial}{\partial z_1}=\frac{\partial}{\partial x_1}-i\frac{\partial}{\partial x_2}\,,\quad
\frac{\partial}{\partial z_2}=\frac{\partial}{\partial x_3}-i\frac{\partial}{\partial x_4}\,,\]
 $D$ and \(D^{\dagger}\) have the following coordinate expressions;
\begin{equation}\label{Dirac}
D =  \begin{pmatrix} \frac{\partial}{\partial z_1} & - \frac{\partial}{\partial \Bar{z_2}} 
\\ \\ \frac{\partial}{\partial z_2} & \frac{\partial}{\partial \Bar{z_1}} \end{pmatrix} , 
\quad
D^{\dagger} = \begin{pmatrix} \frac{\partial}{\partial \Bar{z_1}} & \frac{\partial}{\partial \Bar{z_2}} 
\\ \\ - \frac{\partial}{\partial z_2} & \frac{\partial}{\partial z_1} \end{pmatrix}.
\end{equation}
We have also the following  quaternion expression:
\begin{equation}
D = \frac{\partial}{\partial z_1}+j\frac{\partial}{\partial z_2},\quad D^{\dagger} = \frac{\partial}{\partial z_1}-j\frac{\partial}{\partial z_2}.
\end{equation}

\subsection{Harmonic spinors}

\subsubsection{harmonic polynomials on \(S^3\subset\mathbf{C}^2\)}
 An even (resp. odd) spinor $\varphi$ is called a {\it harmonic spinor} if $D \varphi = 0$ 
 ( resp. \(D^{\dagger} \varphi = 0\) ).

We shall introduce a set of harmonic spinors which, restricted to \(S^3\), forms a 
complete orthonormal system of \(L^{2}(S^3, S^{+})\),   \cite{Ko0, Ko1}.
We introduce the following basis of the vector fields on 
 $\{|z|= 1\}\simeq S^3$. 
\begin{eqnarray*}
e_+ &=& -z_2 \frac{\partial}{\partial \Bar{z_1}} +z_1 \frac{\partial}{\partial \Bar{z_2}} ,
\quad
e_- = - \Bar{z_2} \frac{\partial}{\partial z_1} + \Bar{z_1} \frac{\partial}{\partial z_2}
\\[0.2cm]
\theta &=& \left(z_1 \frac{\partial}{\partial z_1} + z_2 \frac{\partial}{\partial z_2}
- \Bar {z_1} \frac{\partial}{\partial \Bar{z_1}} - \Bar{z_2} \frac{\partial}{\partial \Bar{z_2}}\right)
\end{eqnarray*}
We have the commutation relations;
\begin{equation*}
[\theta , e_+] = 2e_+,\quad [\theta , e_-] = -2e_-, \quad [e_+ ,e_-]=- \theta. 
\end{equation*}
We prefere rather the following basis 
\begin{equation}\label{vectbasis}
\theta_0=\sqrt{-1}\theta\,, \quad \theta_1=e_++e_-\,,\quad \theta_2=\sqrt{-1}(e_+-e_-)\,.
\end{equation}
These are real vector fields; \(\overline{\theta_k}=\theta_k\), \(k=0,1,2\), and  are nothing but the vector fields coming from the infinitesimal ( right ) action of \(su(2)\) on \(\mathbf{C}^2\):
\[dR(\sigma_2)=\sqrt{-1}\theta_1\,,\, dR(\sigma_1)=\sqrt{-1}\theta_2\,,\,dR(\sigma_3)=-\sqrt{-1}\theta_0\,.\]

\begin{remark}~~~
By the Euler angle coordinates \((\theta,\phi,\psi)\) on \(S^3\);  \(\,z_1=\cos\frac{\theta}{2}\exp(\frac{\sqrt{-1}}{2}(\psi+\phi))\), \(z_2=\sqrt{-1}\sin \frac{\theta}{2} \exp(\frac{\sqrt{-1}}{2}(\psi-\phi) )\), 
we have the following expression:
\begin{eqnarray*}
\theta_0&=&\frac{\partial}{\partial\psi},\\
 \theta_1&=&\,-\sin\psi\frac{\partial}{\partial\theta}+\frac{\cos\psi}{\sin\theta}\frac{\partial}{\partial\phi}
-\cot\theta\cos\psi\frac{\partial}{\partial\psi}\\
 \theta_2&=&\,\cos\psi\frac{\partial}{\partial\theta}+\frac{\sin\psi}{\sin\theta}\frac{\partial}{\partial\phi}
-\cot\theta\sin\psi\frac{\partial}{\partial\psi}
\end{eqnarray*}
\hfil\qed
\end{remark}

The dual basis are given by the following differential 1-forms:
\begin{eqnarray*}
\theta_0^{\ast}&=&\frac{1}{2\sqrt{-1}|z|^2}( \overline z_1 d z_1+\overline z_2 dz_2 -z_1 d\overline z_1- z_2 d \overline z_2  ),\\[0.2cm]
\theta_1^{\ast}&=&\frac{1}{2|z|^2}(e_+ ^{\ast}+e_- ^{\ast})\,,\qquad
\theta_2^{\ast}=\frac{1}{2\sqrt{-1}|z|^2}(e_+ ^{\ast}-e_- ^{\ast})\,,
\end{eqnarray*}
where
\begin{equation*}
e_+ ^{\ast}=( -\overline z_2 d\overline z_1+\overline z_1 d\overline z_2 )\,,
\quad
e_- ^{\ast}=( -z_2 d z_1+ z_1 d z_2 )\,.
\end{equation*}
It holds  
\(\theta_j^{\ast}(\theta_k)=\delta_{jk}\) for \(j,k=0,1,2\), and these are real 1-forms:
\begin{equation}\label{real}
\overline\theta_k^{\ast}=\theta_k^{\ast}\,,\quad k=0,1,2\,.
\end{equation}
The 
integrablity condition becomes
\begin{equation}\label{integrable}
\frac{\sqrt{-1}}{2}d\theta_0^{\ast}=\theta_1^{\ast}\wedge \theta_2^{\ast}\,,\quad
\frac{\sqrt{-1}}{2}d \theta_1^{\ast}=\theta_2^{\ast}\wedge \theta_0^{\ast}\,,\quad
\frac{\sqrt{-1}}{2}d \theta_2^{\ast}=\theta_0^{\ast}\wedge \theta_1^{\ast}\,,\end{equation}
and 
\(\theta_0^{\ast}\wedge \theta_1^{\ast}\wedge \theta_2^{\ast}=d\sigma_{S^3}\,
\) is the volume form on \(S^3\).

In the following we denote a function $f(z, \Bar{z})$ of variables $z, \Bar{z}$ simply by $f(z)$.
   For \(m = 0, 1, 2, \cdots\), and \( l, k = 0, 1, \cdots , m\),  we define the polynomials:
\begin{eqnarray}
v ^{k} _{(l,m-l)} &=& (e_-)^k z^{l}_{1} z^{m-l}_{2}.\\[0.2cm]  
 w^{k} _{(l,m-l)}&=& (-1)^k\frac{l!}{(m-k)!}\,v^{m-l}_{(k,m-k)}\,.
\end{eqnarray}
Then \(v ^{k} _{(l,m-l)}\) are harmonic polynomials on $\mathbf{C}^2$;   
\(\Delta v ^{k} _{(l,m-l)}=0\,\),  where \(\Delta
= \frac{\partial ^2}{\partial z _1 \partial \Bar{z}_1} + \frac{\partial ^2}{\partial z _2 \partial \Bar{z}_2}
\).      \(\left\{\,\frac{1}{\sqrt{2}\pi}v ^{k} _{(l,m-l)}\, ;  \,m = 0, 1, \cdots,\,  0\leq k,l\leq m\,\right\} \)  forms a \(L^2(S^3)\)-complete orthonormal system of the space of harmonic polynomials.    Similarly for   \(\left\{\,\frac{1}{\sqrt{2}\pi}w ^{k} _{(l,m-l)}\, ;  \,m = 0, 1, \cdots,\,  0\leq k,l\leq m\,\right\} \).

We see that, for each pair \((m,l)\), \(0\leq l\leq m\),  the subspace  \(H_{(m,l)}=\{v ^{k} _{(l,m-l)}\,; 0\leq k\leq m+1\}\) gives a \((m+1)\)-dimensional right representation of \(su(2)\) with the highest weight \(\frac m2\), \cite{Ko0}.
We note that \(\{w^{k} _{(l,m-l)}\,\}\) are the quantities parallel to \(\{v ^{k} _{(l,m-l)}\,\} \) but corresponding to the left action of \(su(2)\).  

The radial vector field is defined by
\begin{equation}
\frac{\partial}{\partial n} = \frac{1}{2 |z|}(\nu + \Bar{\nu}), \qquad \nu = z_1 \frac{\partial}{\partial z_1} + z_2 \frac{\partial}{\partial z_2} .\label{radial}
\end{equation}

\subsubsection{Harmonic spinors on \(S^3\subset\mathbf{C}^2\)}

We shall denote by \(\gamma\) the Clifford multiplication of the radial vector $\frac{\partial}{\partial n}\,$.   
The multiplication  $\gamma$ changes the chirality:
\(\gamma=
\gamma_+\oplus\gamma_- : S^+ \oplus S^- \longrightarrow S^- \oplus S^+ \), and \(\gamma^2 = 1
\).
\begin{proposition}\cite{Ko1}~~~
The Dirac operators $D$ and \(D^{\dagger}\) have the following polar decompositions:
\begin{eqnarray*}
D &=& \gamma_+ \left( \frac{\partial}{\partial n} - \Do \right) ,\\[0.2cm]
D^\dagger &=& \left( \frac{\partial}{\partial n} + \Do + \frac{3}{2|z|} \right)\gamma_- \,,
\end{eqnarray*}
where the tangential (nonchiral) Dirac operator $\Do$ is given by
\begin{equation*}
\Do = - \left[ \sum^{3} _{i = 1} \left( \frac{1}{|z|} \theta_i \right) \cdot \nabla_{\frac{1}{|z|} \theta_i} \right]
= \frac{1}{|z|} 
\begin{pmatrix}
-\frac{1}{2} \theta & \,e_+ \\[0.2cm]
-e_- &\, \frac{1}{2} \theta
\end{pmatrix}.
\end{equation*}
\end{proposition}
The tangential Dirac operator on the sphere \(S^3 = \{|z| = 1 \}\);
 \begin{equation*}
 \Do | S^3 : C^{\infty} (S^3, S^+) \longrightarrow C^{\infty} (S^3, S^+)
 \end{equation*}
is a self adjoint elliptic differential operator.

Now we introduce a basis of the space of even harmonic spinors by the following formula.   
For $m = 0,1,2, \cdots ; l = 0,1, \cdots , m$ and $k=0,1, \cdots , m+1$, we put 
\begin{eqnarray}\label{basespinor}
\phi^{+(m,l,k)} (z) &=& \sqrt{\frac{(m+1-k)!}{k!l!(m-l)!}} \begin{pmatrix} k v^{k-1} _{(l, m-l)}\\ \\ -v^{k}_{(l, m-l)} \end{pmatrix},\\ \notag \\
\phi^{-(m,l,k)} (z) &=& \sqrt{ \frac{(m+1-k)!}{k!l!(m-l)!}} \left(\frac{1}{\vert z\vert^2}\right)^{m+2}\begin{pmatrix} w^{k} _{(m+1-l,l)}\\ \\ w^{k}_{(m-l,l+1)} \end{pmatrix}.
\end{eqnarray}

We have the following 
\begin{proposition}\cite{Ko1}
\begin{enumerate}
\item
\(\phi^{+(m,l,k)}\)  is a harmonic spinor on \(\mathbf{C}^2\) and  \(\phi^{-(m,l,k)}\) is a harmonic spinor on \(\mathbf{C}^2 \backslash \{0\}\) that is regular  at infinity.   
\item
On $S^3 = \{|z| = 1\}$ we have:
 \begin{equation}
 \Do \phi^{+(m,l,k)} =\frac{m}{2} \phi^{+(m,l,k)} \,,\qquad 
 \Do \phi^{-(m,l,k)} = -\frac{m+3}{2} \phi^{-(m,l,k)} \, .
 \end{equation}
 \item
The eigenvalues of $\,\Do$ are 
 \begin{equation}
 \frac{m}{2} \,,\quad - \frac{m+3}{2} \, ; \quad  m = 0, 1, \cdots,
 \end{equation}
and the multiplicity of each eigenvalue is equal to $(m+1)(m+2)$.
\item
The set of eigenspinors
 \begin{equation}
 \left\{ \frac{1}{\sqrt{2}\pi }\phi^{+(m,l,k)}, \quad \frac{1}{\sqrt{2}\pi }\phi^{-(m,l,k)} \,
 ; \quad m = 0, 1, \cdots , \,  0\leq l\leq  m,\, 0\leq k\leq m+1\right\}
 \end{equation}
forms a complete orthonormal system of $L^2 (S^3, S^+)$.
\end{enumerate}
\end{proposition}

\subsection{Algebra of  Laurent polynomial type spinors on \(S^3\)}

The multiplication of two even spinors is defined by
\begin{equation}
\phi_1\cdot \phi_2\,=\,
\,\left(\begin{array}{c}u_1u_2-\overline v_1v_2\\  v_1u_2+\overline u_1v_2\end{array}\right)\,,\end{equation}
for \(\phi=\left(\begin{array}{c}u_i\\ v_i\end{array}\right)\). \(i=1,2\).
This is the corresponding formula to the quaternion multiplication:
\[(u_1+jv_1)(u_2+jv_2)=(u_1u_2-\overline v_1v_2)+j( v_1u_2+\overline u_1v_2).\]
With this multiplication the \(\mathbf{C}\)-vector space \(S^3\mathbf{H}\simeq C^{\infty}(S^3, \,S^+)\) becomes an associative algebra ( not a \(\mathbf{C}\)-algebra).

If \(\varphi\) is a harmonic spinor on 
 \(\mathbf{C}^2\setminus\{0\}\) we have the expansion
 \begin{equation}
 \varphi(z)=\sum_{m,l,k}\,C_{+(m,l,k)}\phi^{+(m,l,k)}(z)+\sum_{m,l,k}\,C_{-(m,l,k)}\phi^{-(m,l,k)}(z),\label{Laurentspinor}
 \end{equation}
that is uniformly convergent on any compact subset of \(\mathbf{C}^2\setminus\{0\}\). 
The coefficients \(C_{\pm(m,l,k)}\) are given by the formula:
\begin{equation}\label{coef}
C_{\pm(m,l,k)}=\,\frac{1}{2\pi^2}\int _{S^3}\, \langle \varphi,\,\phi^{\pm(m,l,k)}\rangle\,d\sigma,
\end{equation}
where \(\langle\,,\,\rangle\) is the inner product of \(S^+\), \cite{Ko2}.    
For later use we mention
\begin{equation}\label{001coefficient}
 \int_{S^3}\,tr\,\varphi \,d\sigma=4\pi^2 Re.C_{+(0,0,1)},
  \end{equation}
\(Re.\) designating the real part.

\begin{definition}
We call the series (\ref{Laurentspinor})  {\it a spinor of  Laurent polynomial type} if only finitely many coefficients \(C_{\pm (m,l,k)}\)  are non-zero .   The space of spinors of Laurent polynomial type is denoted by   \(\mathbf{C}[\phi ^{\pm}] \).  
\end{definition} 

\begin{proposition}\cite{K-I}~~~
The restriction of 
 \( \mathbf{C}[\phi^{\pm}\,]\) to \(S^3\) is an associative subalgebra of \(S^3\mathbf{H}\) generated by the  spinors:
\begin{equation*}
\phi^{ +(0,0,1)}=\left(\begin{array}{c}1\\0\end{array}\right),\,
\phi^{ +(0,0,0)}=\left(\begin{array}{c}0\\-1\end{array}\right),\,
 \phi^{ +(1,0,1)}=\left(\begin{array}{c}z_2\\-\overline z_1\end{array}\right),\,
\phi^{-(0,0,0)}=\left(\begin{array}{c}z_2\\ \overline z_1.\end{array}\right).
\end{equation*}
\end{proposition}

The proof is found in Section 2.4 of \cite{K-I}.   Here we shall give a brief explanation for the reader's convenience.    
In Lemma 4.1 of \cite{Ko0} we proved the product formula for the harmonic polynomials \(v^k_{(a.b)}\,\):   
\[
v ^{k_1} _{(a_1,b_1)} v ^{k_2} _{(a_2,b_2)}=\sum_{j=0}^{a_1+a_2+b_1+b_2}C_j\vert z\vert^{2j}\,v ^{k_1+k_2-j} _{(a_1+a_2-j,\,b_1+b_2-j)} \, ,
\]
for some rational numbers  \(C_j=C_j(a_1,a_2,b_1,b_2,k_1,k_2)\).     
Let \(k=k_1+k_2\), \(a=a_1+a_2\) and \(b=b_1+b_2\).    The above product formula yields that, restricted to \(S^3\), the harmonic  polynomial \(v^k_{(a,b)}\) is equal to  a constant multiple of 
\(\,v^{k_1}_{(a_1,b_1)} v^{k_2}_{(a_2,b_2)}\) modulo a linear combination of polynomials \(v^{k-j}_{(a-j,b-j)}\,\), \(1\leq j\leq min(k,a,b)\).
On the other hand spinors of the form 
\(\left(\begin{array}{c}v^k_{(l,m-l)}\\0\end{array}\right)\) and \(\left(\begin{array}{c}0\\v^{k+1}_{(l,m-l)}\end{array}\right)\) are written by linear combinations of \(\phi^{+(m,l,k+1)}\) and \(\phi^{-(m-1,k,l)}\).    Therefore we find that any product of two spinors  \(\phi^{\pm (m_1,l_1,k_1)}\cdot \phi^{\pm (m_2,l_2,k_2)}\) is written as a linear combination of \(\phi^{\pm(m_1+m_2-n,\cdot,\cdot)}\), \(1\leq n\leq m_1+m_2\).
Therefore  \( \mathbf{C}[\phi^{\pm}]\vert_{S^3}\)  becomes an associative algebra.   Moreover \(\phi^{\pm(m,l,k)}\) is written by a linear combination of the products  \(\,\phi^{\pm(m_1,l_1,k_1)}\cdot \phi^{\pm(m_2,l_2,k_2)}\) for \(0\leq m_1+m_2\leq m-1\,\),  \(0\leq l_1+l_2\leq l\) and \(0\leq k_1+k_2\leq k\) .   Hence 
we find that the algebra  \( \mathbf{C}[\phi^{\pm}]\vert_{S^3}\)  is generated by the following four spinors 
\[
\phi^{ +(0,0,1)}=\left(\begin{array}{c}1\\0\end{array}\right),\quad
\phi^{ +(0,0,0)}=\left(\begin{array}{c}0\\-1\end{array}\right),\quad
 \phi^{ +(1,0,1)}=\left(\begin{array}{c}z_2\\-\overline z_1\end{array}\right),\quad 
\phi^{-(0,0,0)}=\left(\begin{array}{c}z_2\\ \overline z_1\end{array}\right).
\]
 
We must note that  \( \mathbf{C}[\,\phi^{\pm}\,]\) over \(\mathbf{C}^2\setminus \{0\}\) is not an algebra.   

\begin{corollary}
Let \(\sigma\) be  the involution on \(C^{\infty}(S^3,S^+)\) defined  by
\begin{equation}
\sigma\phi=\left(\begin{array}{c}u\\-v\end{array}\right)\,,\quad\mbox{ for } \,\phi=\left(\begin{array}{c}u\\v\end{array}\right)\,.
\end{equation}
Then  \( \mathbf{C}[\phi^{\pm}]\vert_{S^3}\) is invariant by the involution.
\end{corollary}

In fact, since \(\left(\begin{array}{c}v^k_{(l,m-l)}\\0\end{array}\right)\) and \(\left(\begin{array}{c}0\\v^{k+1}_{(l,m-l)}\end{array}\right)\) are written by linear combinations of \(\phi^{+(m,l,k+1)}\) and \(\phi^{-(m-1,k,l)}\), we see that  \(\sigma \phi^{\pm (m-1,k,l)}\in \mathbf{C}[\phi^{\pm}]\).

\subsection{2-cocycles on \(S^3\mathbf{H} \)} 

Let \(S^3\mathbf{H}=Map(S^{3}, \mathbf{H})\,=\,C^{\infty}(S^3, S^+)\) be the set of smooth even spinors on \(S^3\).    We continue to adopt the \(\mathbf{C}\)-linear correspondence (\ref{hccoresp}) : 
\begin{equation}
C^{\infty}(S^3, S^+)\ni \phi=\begin{pmatrix} u\\ v \end{pmatrix}\,\longleftrightarrow\,u+jv\in\,S^3\mathbf{H}\,.\end{equation}

\(S^3\mathbf{H}\) being an associative algebra, we define a \(\mathbf{R}\)-bilinear bracket
 on  \(S^3\mathbf{H}\):
 \begin{equation}
\bigl[\,\phi _1\, , \,\phi _2\,\bigr]=\,\phi_1\cdot\phi_2-\phi_2\cdot\phi_1=\,
 \begin{pmatrix} \,v_1 \Bar{v} _2 - \Bar{v}_1 v_2 \,\\[0.2cm]
\,(u _ 2 - \Bar {u} _2 ) v _1 - (u _1 - \Bar {u} _1) v_2\, \end{pmatrix} \,,
\end{equation} for even spinors   $\phi _1 = \begin{pmatrix} u_1\\ v_1 \end{pmatrix}$ and $\phi _2 = \begin{pmatrix} u_2\\ v_2 \end{pmatrix}$.       This is equivalent to the following \(\mathbf{R}\)-bilinear  bracket: 
 \begin{equation}
\bigl[\,u _1+jv_1, \,u _2+jv_2\,\bigr]\,=\, (v_1 \Bar{v} _2 - \Bar{v}_1 v_2)+\,j((u _ 2 - \Bar {u} _2 ) v _1 - (u _1 - \Bar {u} _1) v_2)\,.
\end{equation}
We are dealing with a \(\mathbf{C}\)-vector space endowed with a \(\mathbf{R}\)-bilinear bracket.   If we define the involution on \(S^3\mathbf{H}\) by 
\begin{equation}
\sigma:\,\phi=\begin{pmatrix} u\\ v \end{pmatrix}\,\longleftrightarrow\,\sigma\phi=\begin{pmatrix} u\\ -v \end{pmatrix}\,,
\end{equation}
\(S^3\mathbf{H}\) becomes a quaternion Lie algebra.

We define the trace of a spinor \(\phi= \begin{pmatrix} u\\ v \end{pmatrix}\) by the formula: 
\begin{equation} 
tr\,\phi\,=\,u+\overline u.
\end{equation}
Evidently we have \(tr\,[\phi,\psi]=0\).

In the following we introduce three 2-cocycles on \(S^3\mathbf{H}\) that come from the base vector fields \(\theta_k\,;\,k=0,1,2\), on \(S^3\), (\ref{vectbasis}).   

For a \(\varphi=\left(\begin{array}{c}u\\ v\end{array}\right)\in S^3\mathbf {H}\), we put 
\begin{equation}
\Theta_{k}\,\varphi\,=\,\frac{1}{2}\,\left(\begin{array}{c}\,\theta_{k}\, u\\[0.3cm] \,\theta_{k}\, v\end{array}\right),\qquad k=0,1,2.
\end{equation}
Note that \(\theta_k\) is a real vector field: \(\,\theta_k=\overline\theta_k\).

\begin{lemma}~~ For any  \(k=0,1,2\), and \(\phi,\,\psi \in S^3\mathbf{H}\), we have 
\begin{eqnarray}
\Theta_{k}\,(\phi\cdot\psi\,)\,&=&\,(\Theta_{k}\,\phi)\cdot \,\psi\,+\,\phi\cdot\,(\Theta_{k}\,\psi)\,.
\label{leibnitz}
\\[0.2cm]
\int_{S^3}\,\Theta_{k}\,\varphi\,d\sigma\,&=&\,0.\label{tracetheta}
\end{eqnarray}
\end{lemma}
\begin{proof}
For the first equation we use the fact \(\overline \theta_k=\theta_k\), (\ref{real}).  
The second assertion follows from the fact 
\begin{equation}\label{vanishing}
\int_{S^3}\,\theta_{k} f\,d\sigma\,=\,0\,,
\end{equation}
for any function \(f\) on \(S^3\).    This is proved as follows.   We consider the 2-form \(\beta=f\theta_1^{\ast}\wedge\theta_2^{\ast}\).   By virtue of the integrable condition (\ref{integrable}) we have 
\[d\beta=(\theta_0f)\,\theta_0^{\ast}\wedge\theta_1^{\ast}\wedge\theta_2^{\ast}=\theta_0f\,d\sigma\,.\]
Hence 
\[0=\int_{S^3}\,d\beta\,=\,\int_{S^3}\theta_0f\,d\sigma.\]
Similarly for the integral of \(\theta_kf\), \(k=1,2\).
\end{proof}

\begin{remark}~~~  
The formula (\ref{vanishing}) is an evident fact if we recognize the invariance under the action of \(SO(4)\) of each \(\theta_k\) and the volume form \(d\sigma\) .   This is noticed to me by Professor T. Iwai of Kyoto University.
\end{remark}

\begin{definition}~~~
For \(\phi_1\) and \(\phi_2\in S^3\mathbf{H}\,\), we put 
\begin{equation}
c_k(\phi_1,\phi_2)\,=\,
\,\frac{1}{2\pi^2}\int_{S^3}\,\,tr\,(\,\Theta_k \phi_1\cdot \phi_2\,)\, d\sigma\,,\quad k=0,1,2\,.
  \end{equation}
 \end{definition}

\begin{proposition}\label{2cocycle}~~ 
For each \(k=0,1,2\), 
\(c_k\) defines a non-trivial ( real valued ) 2-cocycle on the \(\mathbf{C}\)-algebra \(S^3\mathbf{H}\,\).   That is, \(c_k\)  satisfies the equations:
\begin{eqnarray}
&&c_k(\phi_1\,,\,\phi_2)\,=\,-\, c_k(\phi_2\,,\,\phi_1)\,,\label{asym}\\[0.3cm]
&&c_k(\phi_1\cdot\phi_2\,,\,\phi_3)+c_k(\phi_2\cdot\phi_3\,,\,\phi_1\,)+c_k(\phi_3\cdot\phi_1\,,\,\phi_2\,)=0 ,\label{cycle}
\end{eqnarray}
for any \(\phi_1,\,\phi_2,\,\phi_3\in S^3\mathbf{H}\), 
and there is no 1-cochain \(b\) such that \(c_k(\phi_1\,,\phi_2)=b(\,[\phi_1,\,\phi_2]\,)\).
\end{proposition}
\begin{proof}~~~ Evidently each  \(c_k\) is \(\mathbf{R}\)-bilinear ( It is not \(\mathbf{C}\)-bilinear ).
By (\ref{tracetheta}) and the Leibnitz rule (\ref{leibnitz}) we have
\begin{eqnarray*}
0&=&\, \int_{S^3}\,\,tr\,(\,\Theta_k\,(\phi_1\cdot\phi_2)\,)\,d\sigma\,=\,\int_{S^3}\,tr\,\left(\,\Theta_k\,\phi_1\,\cdot\phi_2\,\right)d\sigma\,+\,\int_{S^3}\,tr\,\left(\,\phi_1\cdot\,\Theta_k\,\phi_2\,\right) d\sigma
 \end{eqnarray*}
Hence
 \(\,c_k(\phi_1\,,\,\phi_2\,)\,+\,c_k(\,\phi_2\,,\,\phi_1\,)\,=0\,\).
  The following calculation proves (\ref{cycle}).   
\begin{eqnarray*}
c_k(\phi_1\cdot\phi_2\,,\,\phi_3)&=&\, \int_{S^3}\,\,tr\,(\,\Theta_k(\,\phi_1\cdot\phi_2\,)\cdot\,\phi_3\,)\,d\sigma
 \\[0.2cm]
&=&\, \int_{S^3}\,\,tr\,(\,\Theta_k\phi_1\cdot \phi_2\cdot \phi_3\,)d\sigma \,+\,\, \int_{S^3}\,\,tr\,(\,\Theta_k\phi_2\cdot\,\phi_3\,\cdot\phi_1\,)d\sigma \\[0.2cm]
&=&\,c_k(\phi_1\,,\,\phi_2\cdot\phi_3\,)+c_k(\phi_2\,,\,\phi_3\cdot\phi_1\,)
=\,- c_k(\phi_2\cdot\phi_3\,,\phi_1\,)\,-\,c_k(\phi_3\cdot\phi_1\,,\,\phi_2)
.
\end{eqnarray*}  
Suppose now that \(c_0\) is the coboundary of a 1-cochain 
\(b:\, S^3\mathbf{H}\longrightarrow \mathbf{C}\).    Then 
\begin{equation}
c_0(\phi_1,\,\phi_2)=(\delta\,b)(\phi_1,\,\phi_2)\,=\,b(\,[\phi_1,\,\phi_2]\,)\nonumber
\end{equation}
for any \(\phi_1,\phi_2\in S^3\mathbf{H}\).   
Take 
\(\phi_1
=\,\frac{1}{\sqrt{2}}\phi^{+(1,1,2)}\,
=\left(\begin{array}{c} -\overline z_2 \\  0\end{array}\right)\)
and \(\,\phi_2=\frac{1}{2}( \phi^{+(1,0,1)}+\phi^{-(0,0,0)})=\left(\begin{array}{c}z_2\\ 0\end{array}\right)\) .   Then \([\,\phi_1,\,\phi_2\,]=0\),   
so 
\((\delta b)(\phi_1,\phi_2)=0\).
But \(c_0(\phi_1,\phi_2)=\frac{1}{2}\,\).   Therefore \(c_0\) can not be a coboundary.    For \(\phi_1\) and \(\phi_3=\phi^{+(1,0,2)}=\sqrt{2}\left(\begin{array}{c}\overline z_1\\ 0\end{array}\right)\), we have \([\phi_1,\phi_3]=0\) and \(c_1(\phi_1,\phi_3)=-\frac{1}{\sqrt{2}}\).   So \(c_1\) can not be a coboundary by the same reason as above.    Similarly for \(c_2\).
  \end{proof}

 {\bf Examples}
 \begin{enumerate}
 \item
 \[  c_0(\phi^{\pm(m,l,k)},\,\phi^{\pm(p,q,r)})=0 ,\quad 
c_0(\,\phi^{+(1,1,2)}\,,\,\sqrt{-1}(\phi^{+(1,0,1)}+\phi^{-(0,0,0)})\,)
\,=\sqrt{2}.
\]
  In Lemma 2.13 of \cite{K-I} we listed the values of 2-cocycles \(c_0(\phi^{\pm(m,l,k)},\,\sqrt{-1}\phi^{\pm(p,q,r)})\).   
  \item~~~Let
   \[\kappa=\phi^{+(1,0,1)}=\left(\begin{array}{c} z_2\\-\overline{z}_1\end{array}\right),\quad
\kappa_{\ast}=\,\frac{-\sqrt{-1}}{\sqrt{2}}(\phi^{-(0,0,0)}-\phi^{+(1,1,2)}-\phi^{+(1,0,1)})
=\sqrt{-1}\left(\begin{array}{c}  \overline z_2 \\ \overline z_1\end{array}\right)
.\]
Then 
\[(\Theta_0\,\kappa)\cdot \kappa_{\ast}\,=\,-\frac{1}{2}\left(\begin{array}{c}1\\0\end{array}\right),\]
and
\[c_0(\kappa,\,\kappa_{\ast})=\frac{1}{2\pi^2}\int_{S^3}\,tr\,[(\Theta_0\kappa)\cdot\kappa_{\ast}\,]\,d\sigma_3=-1.\]
Similarly
\[c_1(\kappa,\,\kappa_{\ast})=c_2(\kappa,\,\kappa_{\ast})=0\]
\end{enumerate}

 \subsection{Radial derivative on \(S^3\mathbf{H}\)}

 We define the following operator \(\mathbf{n}_0\) on \(C^{\infty}(S^3)\):
\begin{equation}
\mathbf{n}_0\,f (z)= |z|\frac{\partial}{\partial n}f(z)\,=\,\frac{1}{2}(\nu+\Bar{\nu})f(z)\,.\label{d_0}
\end{equation}
Here we consider the radial derivative of a function on \(\mathbf{C}^2\) and then restrict it to \(S^3=\{|z|=1\}\).

For an even spinor \(\varphi=\left(\begin{array}{c}u\\ v\end{array}\right)\) 
we put 
\begin{equation}\label{radialder}
\mathbf{n}_0\,\varphi\,=\,\left(\begin{array}{c}\mathbf{n}_0\,u\\[0.2cm] \mathbf{n}_0\,v\end{array}\right).
\end{equation}
The radial derivative \(\mathbf{n}_0\) preserves \(\mathbf{C}[\phi^{\pm}]\), that is, \(\,\mathbf{n}_0\varphi\in \mathbf{C}[\phi^{\pm}]\) for \(\varphi\in \mathbf{C}[\phi^{\pm}]\).

\begin{proposition}\label{typepreserve}~~
\begin{enumerate}
\item
\begin{equation}
\mathbf{n}_0( \phi_1 \cdot \phi_2 ) =(\mathbf{n}_0 \phi _1) \cdot \phi_2 + \phi _1 \cdot (\mathbf{n}_0 \phi _2)\,.
\label{leibnitz1}
\end{equation}
\item
\begin{equation}
\mathbf{n}_0 \phi^{+(m,l,k)} = \frac{m}{2}\, \phi^{+(m,l,k)},\quad 
\mathbf{n}_0 \phi^{-(m,l,k)}=\,- \frac{m+3}{2}\, \phi^{-(m,l,k)}.\label{normalder}
\end{equation}
\item
Let \(\varphi\) be a spinor of Laurent polynomial type:
\begin{equation*}
 \varphi(z)=\sum_{m,l,k}\,C_{+(m,l,k)}\phi^{+(m,l,k)}(z)+\sum_{m,l,k}\,C_{-(m,l,k)}\phi^{-(m,l,k)}(z).
 \end{equation*}
 Then \(\mathbf{n}_0\varphi\) is a  spinor of Laurent polynomial type and 
 \begin{equation}
 \int_{S^3}\,tr\,(\,\mathbf{n}_0\,\varphi\,)\,d\sigma\,=\,0\,.
 \end{equation}
\end{enumerate}
\end{proposition}

\begin{proof}~~~
The formula (\ref{normalder}) follows from the definition (\ref{basespinor}).     The last assertion follows from (\ref{001coefficient}) and the fact that 
 the coefficient of \(\phi^{+(0,0,1)} \) in  the Laurent expansion of \(\mathbf{n}_0\varphi\)   vanishes.
 \end{proof}
 
\begin{proposition}\label{deriv}
Let \(c\) denote one of the 2-cocycles \(c_k\), \(k=0,1,2\).   Then we have 
\begin{equation}
c(\,\mathbf{n}_0\phi_1\,,\phi_2\,)\,+\,c(\,\phi_1\,, \mathbf{n}_0\phi_2\,)\,=\,0\,.
\end{equation}
\end{proposition}

\begin{proof}~~~
Since \(\,\theta\,\mathbf{n}_0\,=\,(\nu-\Bar{\nu})(\nu+\Bar{\nu})=\nu^2-\Bar{\nu}^2=\,\mathbf{n}_0\,\theta\,\), we have
\begin{eqnarray*}
0&=&\int_{S^3}\,tr\,(\,\mathbf{n}_0(\Theta\phi_1\cdot\phi_2)\,)\,d\sigma=\int_{S^3}\,tr\,(\,(\mathbf{n}_0\Theta\phi_1)\cdot \phi_2+\Theta\phi_1\cdot \mathbf{n}_0\phi_2\,)\,d\sigma\\[0.2cm]
&=&\int_{S^3}\,tr\,((\Theta\, \mathbf{n}_0\,\phi_1)\cdot\phi_2\,)\,d\sigma+\int_{S^3}\,tr\,(\Theta\phi_1\cdot \mathbf{n}_0\phi_2\,)\,d\sigma\\[0.2cm]
&=& c(\mathbf{n}_0\phi_1,\phi_2)+c(\phi_1,\mathbf{n}_0\phi_2)\,.
\end{eqnarray*}
   \end{proof}

 \subsection{Homogeneous decomposition of \(\mathbf{C}[\phi^{\pm}]\)}

Let  \(\mathbf{C}[\phi^{\pm};\,N]\) be the subspace of \(\mathbf{C}[\phi^{\pm}]\) consisting of those elements that are of homogeneous degree \(N\): \(\varphi(z)=|z|^N\varphi(\frac{z}{|z|})\).   
\(\mathbf{C}[\phi^{\pm};\,N]\) is spanned by the spinors  
\(\varphi=\phi_1\cdots\phi_n \) such that each \(\phi_i\) is equal to  \(\phi_i=\phi^{+(m_i,l_i,k_i)}\) or  \(\,\phi_i=\phi^{-(m_i,l_i,k_i)}\) , where \(m_i\geq 0\) and \(0\leq l_i\leq m_i+1, \,0\leq k_i\leq m_i+2\), and such that 
\[N=\sum_{i:\,\phi_i=\phi^{+(m_i,l_i,k_i)}}\,m_i\,\,-\,\sum_{i:\,\phi_i=\phi^{-(m_i,l_i,k_i)}}\,(m_i+3).\]
It holds that \(\mathbf{n}_0\varphi=\frac{N}{2}\varphi\), so the eigenvalues of \(\mathbf{n}_0\) on \(\mathbf{C}[\phi^{\pm}]\) are \(\left\{\frac{N}{2};\,N\in\mathbf{Z}\,\right\}\) and 
\(\mathbf{C}[\phi^{\pm};\,N]\) is the space of eigenspinors for the eigenvalue \(\frac{N}{2}\).    

{\bf Example }
\[\phi=\phi^{+(2,00)}\cdot\phi^{-(0,00)}\in \mathbf{C}[\phi^{\pm};-1]\,,\quad\mbox{ and }\, \mathbf{n}_0\phi=-\frac12\phi\,.\]
We note that \(-\frac12\) is not an eigenvalue of \(\Do\).

We have the eigenspace decomposition by the radial operator \(\mathbf{n}_0\):
\begin{equation}
\mathbf{C}[\phi^{\pm}]\,=\, \bigoplus_{N\in \mathbf{Z}}\,\mathbf{C}[\phi^{\pm};\,N]\,.
\end{equation}

\section{Extensions of the Lie algebra   \(\mathbf{C}[\phi ^{\pm}] \otimes_{\mathbf{C}} \mathfrak{gl}(n,\mathbf{C})\)}

\subsection{Extension of  the Lie algebra \(Map( S^3, \mathfrak{gl}(n,\mathbf{H}))\) }

Let \(\mathfrak{gl}(n,\mathbf{H})\) be the algebra of \(n\times n\)-matrices with entries in \(\mathbf{H}\).    
  We put 
\begin{eqnarray}\label{mj2nc}
MJ(2n,\mathbf{C})&=&\left\{Z\in \mathfrak{gl}(2n,\mathbf{C})\,,\quad JZ=\overline ZJ\right\}\\[0.2cm]
&=& \left\{\left(\begin{array}{cc}A& -\overline B\\[0.2cm]
B&\overline A\end{array}\right)\,;\quad A,\,B\in \mathfrak{gl}(n,\mathbf{C})\,\right\}.\nonumber
\end{eqnarray}
where 
\[J=\left(\begin{array}{cc}0&I_n\\[0.2cm]-I_n&0\end{array}\right)\,\in\,GL(2n,\mathbf{R}).\]
Then \(\mathfrak{gl}(n,\mathbf{H})\) and \(MJ(2n,\mathbf{C})\) are \(\mathbf{C}\)-isomorphic matrix algebra:
\begin{equation}\label{Cliso}
\mathfrak{gl}(n,\mathbf{H})\,\simeq \,MJ(2n,\mathbf{C})\,,\quad A+JB\longrightarrow 
\left(\begin{array}{cc}A& -\overline B\\[0.2cm]
B&\overline A\end{array}\right)\,.
\end{equation}
Now the \(\mathbf{C}\)-vector spaces \(MJ(2n,\mathbf{C})\) and \(MJ(2,\mathbf{C})\otimes_{\mathbf{C}}\mathfrak{gl}(n,\mathbf{C})\) are isomorphic.   The \(\mathbf{C}\)-linear isomorphism is given by the following transformation of \(2n\times 2n \)-matrices: 
\begin{eqnarray}
&MJ(2,\mathbf{C})\otimes_{\mathbf{C}}\mathfrak{gl}(n,\mathbf{C})\,\ni\,
\sum_{i,j=1}^n\left\{\,a_{ij}E^{\prime}_{2i-1\,2j-1}\,+\,b_{ij}E^{\prime}_{2i\,2j-1}\,-\,\overline b_{ij}E^{\prime}_{2i-1\,2j}\,+\,
\overline a_{ij}E^{\prime}_{2i\,2j}\,\right\}\nonumber \\[0.2cm]
&\longrightarrow \sum_{i,j=1}^n\,\left\{\,a_{ij}E_{ij}+\,b_{ij}E_{n+i,j}\,-\,\overline b_{ij}E_{i,n+j}+\overline a_{ij}E_{n+i,n+j}\right\}\in\,MJ(2n,\mathbf{C})\,.
\end{eqnarray}
Where the \(2n\times 2n \)-matrix base of \(MJ(2,\mathbf{C})\otimes_{\mathbf{C}}\mathfrak{gl}(n,\mathbf{C})\)  is written by
\begin{eqnarray*}
E^{\prime}_{2i-1\,2j-1}&=\left(\begin{array}{cc}1&0\\0&0\end{array}\right)\otimes E_{ij},\quad E_{2i-1\,2j}^{\prime}&=
 \left(\begin{array}{cc}0&1\\0&0\end{array}\right)\otimes E_{ij}, \\[0.2cm]
  E_{2i\,2j-1}^{\prime}&=
 \left(\begin{array}{cc}0&0\\1&0\end{array}\right)\otimes E_{ij}
, \quad E_{2i\,2j}^{\prime}&=\left(\begin{array}{cc}0&0\\0&1\end{array}\right)\otimes E_{ij}\,.
\end{eqnarray*}
From (\ref{qtomj}) and (\ref{Cliso}) we have  the following \(\mathbf{C}\)-vector space isomorphisms:
\begin{equation}\label{c}
\mathfrak{gl}(n,\mathbf{H})\,\simeq\,MJ(2n,\mathbf{C})\simeq MJ(2,\mathbf{C})\otimes_{\mathbf{C}}\mathfrak{gl}(n,\mathbf{C})\simeq \mathbf{H}\otimes_{\mathbf{C}}\mathfrak{gl}(n,\mathbf{C})\,.\end{equation}
By the matrix multiplication each space becomes an associative algebra.     

We define the following bracket on \( \mathfrak{gl}(n,\mathbf{H})\):
 \begin{eqnarray}
 \left[ X_1+JY_1,\,X_2+JY_2\,\right]\,&=&\,(X_1X_2-X_2X_1- \overline Y_1Y_2 + \overline Y_2Y_1)
 \nonumber \\[0.2cm]
\quad  &&+J(Y_1X_2-Y_2X_1+\overline X_1Y_2-\overline X_2Y_1)\,,\label{bracket1}
 \end{eqnarray}
 for \(X_1+JY_1,\,X_2+JY_2\in\mathfrak{gl}(n,\mathbf{H})\).    By (\ref{c}) it has the description    
 \begin{equation}\label{rLbracket}
 \left[\,z_1\otimes X\,,\,z_2\otimes Y\,\right]\,=\, (z_1z_2)\otimes XY\,-\,(z_2z_1)\otimes YX\,,
 \end{equation}
 for any bases \(X,\,Y\) of \(\mathfrak{gl}(n,\mathbf{C})\) and \(z_1,\,z_2\in\mathbf{H}\).   
 More conveniently,   if we write  \(E_{ij}\)  the \(n\times n\)-matrix with entry \(1\) at \((i,j)\)-place and \(0\) otherwise, 
\(\{E_{ij}\}_{i,j}\) is the basis  of \(\mathfrak{gl}(n,\mathbf{C})\) and we have
  \begin{equation}
 \left[\,z_1\otimes E_{ij}\,,\,z_2\otimes E_{kl}\,\right]\,=\, (z_1z_2)\otimes \delta_{jk}E_{il}\,-\,(z_2z_1)\otimes \delta_{il}E_{kj}\,,
 \end{equation}
 for \(z_1, z_2\in\mathbf{H}\).     
It is easy to see that thus defined \(\mathbf{R}\)-linear bracket satisfies the antisymmetry equation and the Jacobi identity.    From (\ref{bracket1}) we see that the involution \(\sigma\) acts on \( \mathfrak{gl}(n,\mathbf{H})\).   Therefore \( \mathfrak{gl}(n,\mathbf{H})\) is a quaternion Lie algebra.      

 We proceed to the algebra of \(\mathfrak{gl}(n,\mathbf{H})\)-valued mappings over \(S^3\): 
   \begin{equation}
 S^3\mathfrak{gl}(n,\mathbf{H})\,=\,   Map( S^3,\,\mathfrak{gl}(n,\mathbf{H})\,)\,=\,S^3\mathbf{H}\otimes_{\mathbf{C}} \mathfrak{gl}(n,\mathbf{C})\,.
 \end{equation}
    Any element is denoted by  
\(F+JG\in S^3\mathfrak{gl}(n,\mathbf{H})\) with \(F,\,G\in S^3\mathfrak{gl}(n,\mathbf{C})\): 
\begin{equation}
 S^3\mathfrak{gl}(n,\mathbf{H})\,=\,S^3\mathfrak{gl}(n,\mathbf{C})\,+\,J(S^3\mathfrak{gl}(n,\mathbf{C}))\,.
 \end{equation}
Or we denote rather  \(\,\phi\otimes X\in S^3\mathfrak{gl}(n,\mathbf{H})\) with \(\phi\in S^3\mathbf{H}\) and \(X\in \mathfrak{gl}(n,\mathbf{C})\).
Thus we see that 
\( S^3\mathfrak{gl}(n,\mathbf{H})\)  endowed with the following bracket  \(\,[\,\,,\,\,]_{S^3\mathfrak{gl}(n,\mathbf{H})}\) becomes a quaternion Lie algebra.
 \begin{equation}\label{glbracet}
 [\,\phi \otimes E_{ij}\, , \,\psi  \otimes E_{kl}\,]_{ S^3\mathfrak{gl}(n,\mathbf{H})} =  ( \phi\cdot\psi ) \otimes \delta_{jk}E_{il} \,
- (\psi \cdot \phi ) \otimes \delta_{il}E_{kj}
 , 
\end{equation}
for \(\phi,\,\psi\,\in S^3\mathbf{H}\simeq C^{\infty}(S^3, S^+) \) .   

We will write the invariant bilinear form ( Killing form ) on  \( \mathfrak{gl}(n,\mathbf{C})\) by 
 \[(X \vert Y)=\,Trace\,(XY) .\]  
 We have  \((XY\vert Z)=(YZ\vert X)\).   
 It holds that  \((E_{ij}\vert E_{kl})=\delta_{ij}\delta_{jk}\).
Then we define \(\mathbf{C}\)-valued 2-cocycles on the Lie algebra   \( S^3\mathfrak{gl}(n,\mathbf{H})\,\) by 
\begin{equation}\label{trcocycle}
\tilde{c}_k(\,\phi_1\otimes X\,,\,\phi_2\otimes Y\,)=\,(X \vert Y)\,c_k(\phi_1,\phi_2)\,,\quad k=0,1,2.
\end{equation}
The 2-cocycle property follows from  Proposition  
\ref{2cocycle}.
  
 Let \(a_k\), \(k=0,1,2\), be three indefinite  numbers.    For each \(k=0,1,2\) there is a central extension of the Lie algebra   \( S^3\mathfrak{gl}(n,\mathbf{H})\,\) by the 1-dimensional center \(\mathbf{C}a_k\) associated to the 2-cocycle \(\tilde{c}_k\).  Here we define \(\sigma a_k=a_k\).   Summing-up the above we arrive at the following theorem.
 
 \begin{theorem}~~ 
 \begin{equation}
S^3\mathfrak{gl}(n,\mathbf{H})(a)\,=\, (\, S^3\mathbf{H} \otimes_{\mathbf{C}}  \mathfrak{gl}(n,\mathbf{C})\,)\oplus (\oplus_{k=0,1,2}\mathbf{C}a_k) 
\end{equation}
endowed with the following bracket is a quaternion Lie algebra.    
 \begin{eqnarray}
 [\,\phi \otimes X\, , \,\psi  \otimes Y\,]^{\,\widehat{\,}}
  &=&   ( \phi\cdot\psi ) \otimes XY \,
- (\psi \cdot \phi ) \otimes YX
+(X\vert Y)\, \sum_{k=0}^2\,\tilde{c}_k(\phi , \psi )\,a_k \,,\nonumber
\\[0,3cm] 
 [\,a_k\,, \,\phi\otimes X\,] ^{\,\widehat{\,}}&=&0\, ,\,\,k=0,1,2
  \end{eqnarray}
for  \(\phi,\,\psi\,\in S^3\mathbf{H} \, \) and any bases \(X,\,Y\in\mathfrak{gl}(n,\mathbf{C})\). 
\end{theorem}

\medskip

As a Lie subalgebra of \(S^3\mathfrak{gl}(n,\mathbf{H})\) we have the Lie algebra of \(\mathfrak{gl}(n,\mathbf{C})\)-valued Laurent polynomial type spinors: \(\,\mathbf{C}[\phi^{\pm}]\otimes_{\mathbf{C}}  \mathfrak{gl}(n,\mathbf{C})\).   The basis of  \(\mathbf
{C}\)-vector space \(\mathbf{C}[\phi^{\pm}]\otimes_{\mathbf{C}}  \mathfrak{gl}(n,\mathbf{C})\) consists of 
\begin{equation}
\phi^{\pm(m,l,k)}\otimes E_{ij}\,.
\end{equation}

\begin{definition}\label{currentgl}~~
\begin{enumerate}
\item We define a \(\mathbf{C}\)-vector space  
\begin{equation}
\widehat{\mathfrak{gl}(n,\mathbf{H})}\,=\,\mathbf{C}[\phi^{\pm}]\otimes_{\mathbf{C}}  \mathfrak{gl}(n,\mathbf{C})\,.
\end{equation}
\(\widehat{\mathfrak{gl}(n,\mathbf{H})}\) is endowed with the Lie bracket (\ref{rLbracket}).
\item
We denote by \(\,\widehat{\mathfrak{gl}(n,\mathbf{H})}(a)\,\) the extension of the Lie algebra \(\,\widehat{\mathfrak{gl}(n,\mathbf{H})}\,\) by the three 1-dimensional centers \(\mathbf{C}a_k\) associated to the cocycle \(c_k\,\):
\begin{equation}\label{extension}
\widehat{\mathfrak{gl}(n,\mathbf{H})}(a)=\mathbf{C}[\phi^{\pm}]\otimes_{\mathbf{C}}  \mathfrak{gl}(n,\mathbf{C})\,\oplus (\oplus_{k=0,1,2}\mathbf{C}a_k).
\end{equation}
This is a direct sum of \(\mathbf{C}\)-vector spaces.
 \end{enumerate}
\end{definition} 

\begin{proposition}
\(\widehat{\mathfrak{gl}(n,\mathbf{H})}\) and \(\,\widehat{\mathfrak{gl}(n,\mathbf{H})}(a)\,\) are quaternion Lie algebras.
\end{proposition}

\subsection{The outer derivation of \(\widehat{\mathfrak{gl}(n,\mathbf{H})}(a)\)}

We introduced the radial derivative \(\mathbf{n}_0\) acting on \(S^3\mathbf{H}\) in (\ref{radialder}).   \(\mathbf{n}_0\) preserves the space of spinors of Laurent polynomial type \(\mathbf{C}[\phi^{\pm}]\), Proposition\ref{typepreserve}.   
The derivation \(\mathbf{n}_0\) on \(\mathbf{C}[\phi^{\pm}]\) is extended to an outer derivation of the Lie algebra \(\,\widehat{\mathfrak{gl}(n,\mathbf{H})}\,=\mathbf{C}[\phi^{\pm}]\otimes_{\mathbf{C}} \mathfrak{gl}(n,\mathbf{C})\,\):
\begin{equation}
   \,\mathbf{n}_0\,(\phi \otimes X\,)\,=\,(\mathbf{n}_0\phi \,)\otimes X,\qquad \,\phi\in \mathbf{C}[\phi^{\pm}],\,X\in \mathfrak{gl}(n,\mathbf{C})\,.
\end{equation}
In fact we have 
\begin{eqnarray*}
&&\,[\,\mathbf{n}_0(\phi_1\otimes X_1)\,,\,\phi_2\otimes X_2\,]^{\,\widehat{\,}}\,+\,
[\,\phi_1\otimes X_1\,,\,\mathbf{n}_0(\phi_2\otimes X_2)\,]^{\,\widehat{\,}}
  \\[0.2cm]
 && \,=\,(\mathbf{n}_0\phi_1 \cdot\phi_2)\otimes(X_1X_2)\,-\,(\phi_2\cdot \mathbf{n}_0\phi_1)\otimes (X_2X_1)
\,
 +\,(\phi_1 \cdot \mathbf{n}_0\phi_2)\otimes (X_1X_2)\,\\[0.2cm]
&&
 \qquad -\,(\mathbf{n}_0\phi_2\cdot \phi_1)\otimes (X_2X_1)\,   +\,(X_1\vert X_2)\sum_k\left(\,\tilde{c}_k(\mathbf{n}_0\phi_1,\,\phi_2)\,+\,\tilde{c}_k(\phi_1,\,\mathbf{n}_0\phi_2)\right)a_k\,.
\end{eqnarray*}  
 Since \(c_k(\mathbf{n}_0\phi_1,\,\phi_2)\,+\,c_k(\phi_1,\,\mathbf{n}_0\phi_2)=0\) from Proposition \ref{deriv}, the right-hand side  is equal to
\begin{equation*}
  \mathbf{n}_0\left(\,[\,\phi_1\otimes X_1\,,\,\phi_2\otimes X_2\,]^{\,\widehat{\,}}\,\,\right)
  \end{equation*}
by virtue of (\ref{leibnitz1}).    Hence  
 \begin{equation*}
\mathbf{n}_0\left(\,[\,\phi_1\otimes X_1\,,\,\phi_2\otimes X_2\,]^{\,\widehat{\,}}\,\right)
=
 [\,\mathbf{n}_0(\phi_1\otimes X_1)\,,\,\phi_2\otimes X_2\,]^{\,\widehat{\,}}\,+ 
[\,\phi_1\otimes X_1\,,\,\mathbf{n}_0(\phi_2\otimes X_2)\,]^{\,\widehat{\,}}\,.
\end{equation*}
Therefore \(\mathbf{n}_0\) is a derivation that acts on the Lie algebra \(\mathbf{C}[\phi^{\pm}]\otimes_{\mathbf{C}} \mathfrak{gl}(n,\mathbf{C})\,\).    \\

  We denote by \(\widehat{\mathfrak{gl}}\,\) the Lie algebra that is obtained by adjoining the outer  derivation \(\mathbf{n}_0\) to  \(\widehat{\mathfrak{gl}(n,\mathbf{H})}(a)\).   We let \(\mathbf{n}_0\) kill 
  \(a_k\), \(k=0,1,2\).
   More precisely we have the following
\begin{theorem}~~~
 Let \(a_k\), \(k=0,1,2\), and  \(\,\mathbf{n}\) be indefinite elements. We consider the \(\mathbf{C}\) vector space:
\begin{equation}
\widehat{\mathfrak{gl}}\,=\, \left(\mathbf{C}[\phi^{\pm}]\otimes_{\mathbf{C}} \mathfrak{gl}(n,\mathbf{C})\right) \oplus( \oplus_{k=0}^2\,\mathbf{C}\, a_k )\oplus (\mathbf{C}\mathbf{n})\,,
\end{equation}
 and define the following \(\mathbf{R}\)-linear bracket on $\widehat{\mathfrak{gl}}\):   
  For \(X,Y\in \mathfrak{gl}(n,\mathbf{C})\) and \(\phi,\,\psi\,\in \,\mathbf{C}[\phi^{\pm}]\, \), 
 \begin{eqnarray}
&& [\,\phi \otimes X\, , \,\psi  \otimes Y\,]_{\widehat{\mathfrak{gl}}} \,=\,
  [\,\phi \otimes X\, , \,\psi  \otimes Y\,]^{\,\widehat{\,}}    \nonumber\\[0.2cm]
  &&\quad =  (\phi\cdot\psi)\otimes\,(XY) - (\psi \cdot \phi) \otimes (YX)+ (X|Y)\, \sum_{k=0}^2\,\tilde{c}_k(\phi , \psi )\,a_k \, , 
\\[0,2cm] 
&&\, [\,a_k\,, \,\phi\otimes X\,] _{\widehat{\mathfrak{gl}}}\,=0\,,\\[0.2cm]
&& [\,\mathbf{n}, \phi \otimes X\,] _{\widehat{\mathfrak{gl}}}=\,\mathbf{n}_0\phi \otimes X\,, \, \, 
   [\,\mathbf{n}\,,\,a_k\,]_{\widehat{\mathfrak{gl}}}\,=0,\quad k=0,1,2\,.
  \end{eqnarray}
We introduce the involution on  \( \widehat{\mathfrak{gl}}\,\) by 
\begin{equation}
\sigma(\,\phi\otimes X)=\sigma\phi\otimes X\,,\quad \sigma a_k=0\,, \quad \sigma \mathbf{n}=\mathbf{n}\,.
\end{equation}
      Then \( \left(\, \widehat{\mathfrak{gl}} \, , \, [\,\cdot,\cdot\,]_{ \widehat{\mathfrak{gl} } } \,\right) \) becomes a quaternion Lie algebra.
\end{theorem}

\begin{proof}~~~ It is enough to prove the following Jacobi identity:
 \begin{equation*}
 [\,[\,\mathbf{n}\,, \,\phi_1   \otimes X_1 \,]_{\widehat{\mathfrak{g}}}\,,\, \phi_2 \otimes X_2\,]_{\widehat{\mathfrak{g}}}
+[\,[\phi_1  \otimes X_1 , \phi_2 \otimes X_2\,]_{\widehat{\mathfrak{g}}}\,,\,\mathbf{n}\,]_{\widehat{\mathfrak{g}}}
\,+\,[\,[\phi_2 \otimes X_2 , \,\mathbf{n}\,]_{\widehat{\mathfrak{g}}} ,\, \phi _1   \otimes X_1\,]_{\widehat{\mathfrak{g}}}=0.
\end{equation*}
In the following we shall abbreviate the bracket \([\,,\,]_{\widehat{\mathfrak{gl}}}\,\) simply to \([\,\,,\,\,]\).    
  We have
\begin{align*}
[\,[\,\mathbf{n}\,, \,\phi_1  \otimes X_1 \,]\,, \phi_2 \otimes X_2\,] =&
[\,\mathbf{n}_0\phi_1\otimes X_1,\,\phi_2 \otimes X_2\,] \\[0.2cm]
=& \,( \,\mathbf{n}_0\phi_1\cdot \phi_2 )\otimes\,\,(X_1X_2 )
-\,(\,\phi_2\cdot \mathbf{n}_0\phi_1\,)\otimes (X_2X_1)\\[0.2cm]
&+
(X_1\vert X_2)\sum \tilde{c}_k(\,\mathbf{n}_0\phi_1\,,\,\phi_2)\, a_k\,.
\end{align*}
Similarly
\begin{align*}
[\,[ \,\phi_2  \otimes X_2, \,\mathbf{n} \,] , \phi_1 \otimes X_1\,]=&
(\phi_1\cdot \mathbf{n}_0\phi_2) \otimes (X_1X_2 ) 
-\,(\,\mathbf{n}_0\phi_2 \cdot\phi_1)\otimes (X_2X_1)\\[0.2cm]
&+
(X_1\vert X_2)\,\sum_k \tilde{c}_k(\phi_1,\,\mathbf{n}_0\phi_2\,)\, a_k\,.  \\[0.2cm]
[\,[\phi_1  \otimes X_1 , \phi_2  \otimes X_2\,]\,,\,\mathbf{n}\,]=&
-\bigl[\,\mathbf{n}\,,\,
(\phi_1\cdot \phi_2 )\otimes (X_1X_2 )-(\phi_2\cdot\phi_1)\otimes (X_2X_1)\\[0.2cm]
&\qquad +(X_1\vert X_2)\sum_k \tilde{c}_k (\phi_1\,, \phi_2)\,a_k\,\bigr]\\[0.2cm]
=&-\,\mathbf{n}_0(\phi_1\cdot\phi_2) \otimes\,(X_1X_2 )
+ \mathbf{n}_0(\phi_2\cdot\phi_1)\otimes (X_2X_1) \,.
\end{align*}
  The sum of three equations  vanishes by virtue of (\ref{leibnitz1}) and Proposition \ref{deriv}.   
 \end{proof}
 \par\medskip
 
 \begin{proposition}
  The centralizer of \(\,\mathbf{n}\) in \(\,\widehat{\mathfrak{gl}}\,\) is given by
\begin{equation}
(\,\mathbf{C}[\phi^{\pm};\,0]\,\otimes_{\mathbf{C}} \mathfrak{gl}(n.\mathbf{C})\,)\,\oplus \,( \oplus_k\mathbf{C}a_k\,)\oplus\mathbf{C}\mathbf{n}\,.
 \end{equation}
 \end{proposition}
 We remember that \(\,\mathbf{C}[\phi^{\pm};\,0]\) is the subspace in \(\mathbf{C}[\phi^{\pm}]\) generated by \(\phi_1\cdots\phi_n\) with \(\phi_i\) being \( \phi_i=\phi^{\pm(m_i,l_i,k_i)} \) such that 
 \[\sum_{i;\,\phi_i=\phi^{+(m_i,l_i,k_i)} }\,m_i-\sum_{i;\,\phi_i=\phi^{-(m_i,l_i,k_i)}}\,m_i=0\,.\]
 
\section{\(\mathfrak{sl}(n,\mathbf{H})\)-current algebras on \(S^3\)}

\subsection{Preliminaries on \(\mathfrak{sl}(n,\mathbf{H})\) }

Let \(\mathfrak{sl}(n, \mathbf{H})\) denote the quaternion special linear algebra:
\begin{equation}
\mathfrak{sl}(n, \mathbf{H})\,=\,\left\{\,Z\in\,\mathfrak{gl}(n, \mathbf{H})\,;\quad tr\,Z=0\,\right\}.
\end{equation}
(\ref{Cliso}) implies that 
   \(\mathfrak{sl}(n, \mathbf{H})\) is \(\mathbf{C}\)-isomorphic to 
\begin{equation}\label{slH}
 \left\{\left(\begin{array}{cc}A&-\overline B\\[0.2cm]
B&\overline A\end{array}\right)\in MJ(2n,\mathbf{C}):\quad tr\,A\in\sqrt{-1}\mathbf{R}\,\right\}\,.
\end{equation}
If we put 
\begin{equation}
\mathfrak{sk}(n,\mathbf{C})=\{A\in \mathfrak{gl}(n,\mathbf{C})\,;\quad tr\,A\in \sqrt{-1}\mathbf{R}\,\},
\end{equation}
then  \( \mathfrak{sl}(n,\mathbf{H})\,=\, \mathfrak{sk}(n,\mathbf{C})+\,J \mathfrak{gl}(n,\mathbf{C})\) 
as a complex \(\mathbf{C}\)-module, and  
we have the \(\mathbf{C}\)-linear isomorphism:
\begin{equation}
\mathfrak{sl}(n,\mathbf{H})\ni A+JB\,\longrightarrow\,\left(\begin{array}{cc}A&-\overline B\\[0.2cm]
B&\overline A\end{array}\right)\in MJ(2n,\mathbf{C}),\quad A\in \mathfrak{sk}(n,\mathbf{C}).
\end{equation}
Let  \(\mathfrak{sl}(n, \mathbf{C})\) be the complex special linear algebra.
We have the following relations of \(\mathbf{C}\)-submodules:
\begin{equation}
\mathbf{H}\otimes_{\mathbf{C}} \mathfrak{sl}(n,\mathbf{C})\,
\stackrel{\flat}{= } \,\mathfrak{sl}(n,\mathbf{C})+\,J\mathfrak{sl}(n,\mathbf{C})\,\subset \mathfrak{sk}(n,\mathbf{C})+\,J\mathfrak{gl}(n,\mathbf{C})=
\mathfrak{sl}(n,\mathbf{H})\,,
\end{equation}
where \(\stackrel{\flat}{= }\) is given by the correspondence (\ref{Cliso}).
\(\mathbf{H}\otimes_{\mathbf{C}} \mathfrak{sl}(n,\mathbf{C})\) is not a Lie algebra but a \(\mathbf{H}\)-submodule of \(\mathfrak{gl}(n,\mathbf{H})\).     
\(\mathfrak{sl}(n,\mathbf{H})=\mathfrak{sk}(n,\mathbf{C})\,+\,J \mathfrak{gl}(n,\mathbf{C})\) is 
a semi-submodule of \(\mathfrak{gl}(n,\mathbf{H})\), and  \(\mathfrak{sl}(n,\mathbf{H})\) is the quaternification of \(\mathfrak{sk}(n,\mathbf{C})\). As for a {\it semi-submodule } and a {\it quaternification} see the discussion above Definition\ref{ql}.     Actually we show in the next Proposition that \(\mathfrak{sl}(n,\mathbf{H})\) is generated, as a Lie subalgebra of \(\mathfrak{gl}(n,\mathbf{H})\),  by \(\mathfrak{sl}(n,\mathbf{C})\) and \(J\mathfrak{sl}(n,\mathbf{C})\).   So \(\mathfrak{sl}(n,\mathbf{H})\) is also adequate to be called the quaternification of \(\mathfrak{sl}(n,\mathbf{C})\), 

Let  \(\mathfrak{h}\) be the diagonal matrices of \(\mathfrak{sl}(n,\mathbf{C})\).    
\(\mathfrak{h}\) is a Cartan subalgebra of  \(\mathfrak{sl}(n,\mathbf{C})\), and  \(ad(H)E_{ij}=(\lambda_i-\lambda_j)E_{ij}\) for \(H=\sum^n\lambda_iE_{ii}\in\mathfrak{h}\) such that \(\,\sum^n\lambda_j=0\).   So the set of roots of \(\mathfrak{sl}(n,\mathbf{C})\) with respect to \(\mathfrak{h}\) consist of the \((n-1)n+1\) elements:
\(\,0\,,\,\Delta=\{\alpha_{ij}\,\}\), 
where \(\alpha_{ij}(H)=\lambda_i-\lambda_j\,\) for any \(H=\sum^n\lambda_iE_{ii}\in\mathfrak{h}\).   
Let  \(\mathfrak{sl}(n,\mathbf{C})= \mathfrak{h}\oplus \sum_{i\neq j}\,\mathfrak{g}_{\alpha_{ij}}\) be the root space decomposition with  \(\,\mathfrak{g}_{\alpha_{ij}}=\mathbf{C}E_{ij}\,\),\(\,i\neq j\)\,.  
     Let \(\Pi=\{\alpha_i=\lambda_i-\lambda_{i+1};\,i=1,\cdots,r=n-1\,\}\subset \mathfrak{h}^{\ast}\) be the set of simple roots
and let  \(\{\alpha_i^{\vee}\,;\,i=1,\cdots,r\,\}\subset \mathfrak{h}\) be the set of simple coroots.   
It holds that
\[ \alpha_{ij}=\alpha_i+\cdots\,+\alpha_{j-1}\,, \mbox{ for } i<j\,,\quad \alpha_{ij}=-(\alpha_j+\cdots\,+\alpha_{i-1})\,, \mbox{ for } j<i\,,\]
The Cartan matrix \(A=(\,a_{ij}\,)_{i,j=1,\cdots,r}\) is given by \(a_{ij}=\left\langle \alpha_i^{\vee},\,\alpha_j \right\rangle\).      Fix a standard set of generators of \(\mathfrak{sl}(n,\mathbf{C})\):
      \[h_i=\alpha_i^{\vee}\,,\,x_i \in \mathfrak{g}_{\alpha_i}\,, \,y_i \in \mathfrak{g}_{-\alpha_i}\,,\quad 1\leq i\leq n-1\,,\]
 so that 
 \begin{equation}\label{slrelation}
 [\,x_i,\,y_j \,]=h_j\delta_{ij}\,, \,[\,h_i,\,x_j\,]=a_{ji}x_j\,, \,[\,h_i,\,y_j\,]=-a_{ji}y_j\,.
 \end{equation}
If \(a_{ij}=0\),  then \( [\,x_i,\,x_j \,]=0\) and  \([\,y_i,\,y_j \,]=0\).     We see that \( \{E_{ij}\,;\,1\leq i\neq \, j\leq n\} \) are  generated as follows;
\begin{eqnarray}\label{compose}
E_{ij}&=&\left [\,x_i\,,\,[x_{i+1},\,\cdots\, [x_{j-2},\,x_{j-1}]\,]\cdots\,\right ],\,\mbox{ for } i<j\,,\nonumber\\[0.2cm]
E_{ij}&=&\,\left [\,y_j\,,\,[y_{j+1},\,\cdots\, [y_{i-2},\,y_{i-1}]\,]\cdots\,\right ],\,\mbox{ for } i>j\,.
\end{eqnarray}
\begin{lemma}
The following relations hold in \(\mathfrak{gl}(n,\mathbf{H})\):
\begin{eqnarray*}
&& [\,x_i,\,Jy_j \,]=Jh_j\delta_{ij}\,, \,[\,h_i,\,Jx_j\,]=a_{ji}Jx_j\,, \,[\,h_i,\,Jy_j\,]=-a_{ji}Jy_j\,.\\[0.2cm]
 && [\,Jx_i,\,y_j \,]=-Jh_j\delta_{ij}\,, \,[\,Jh_i,\,x_j\,]=a_{ji}Jx_j\,, \,[\,Jh_i,\,y_j\,]=-a_{ji}Jy_j\,.\\[0.2cm]
 &&  [\,Jx_i,\,Jy_j \,]= - h_j\delta_{ij}\,, \,[\,Jh_i,\,Jx_j\,]=-a_{ji}x_j\,, \,[\,Jh_i,\,Jy_j\,]=-a_{ji}Jy_j\,.
 \end{eqnarray*}
\begin{eqnarray*}
[\,Jh_i\,,\,E_{jk}\,]\,&=&\,(a_{ij}+\cdots\,a_{i\,k-1})JE_{jk},\quad j<k \\[0.2cm]
[\,Jh_i\,,\,E_{jk}\,]\,&=&\,-(a_{ij}+\cdots\,a_{i\,k-1})JE_{jk},\quad k<j.
\end{eqnarray*}
\end{lemma}

\begin{proposition}\label{generators}
The \(\mathbf{C}\)-module  \(\mathbf{H}\otimes_{\mathbf{C}} \mathfrak{sl}(n,\mathbf{C})=\mathfrak{sl}(n,\mathbf{C})+J\mathfrak{sl}(n,\mathbf{C})\) generates  the Lie algebra \( \mathfrak{sl}(n,\mathbf{H})\) over \(\mathbf{C}\).        The generators are given by 
\begin{equation}
h_i\,,\,x_i\,,\,y_i\,,\, Jh_i\,,\,Jx_i\,,\,Jy_i\,; \quad ( \,i =1,\cdots,n-1\,)\,.
\end{equation}
\end{proposition}
\begin{proof}~~~
Take the basis  \(B=\{\,h_i\,,\,1\leq i\leq n-1\,;\quad E_{j\,k}\,,\, 1\leq j\neq \,k\leq n,\,\}\) of  \(\,\mathfrak{sl}(n,\mathbf{C})\).   The subset \(\{B,\,JB\}\) of \(\,\mathfrak{sl}(n,\mathbf{C})+\,J\mathfrak{sl}(n,\mathbf{C})\)
augmented by the two elements \(\sqrt{-1}\,E_{2\,2}\) and \(J( E_{ii}+E_{nn})\) present a basis of  \(\mathfrak{sl}(n,\mathbf{H})=\mathfrak{sk}(n,\mathbf{C})+J\mathfrak{gl}(n,\mathbf{C})\).   So we show that these two elements are generated by \(B\).
In fact, we have  \(\,[\sqrt{-1}Jh_{1}\,,\,Jh_{2}\,]\,=\, 2\sqrt{-1}\,E_{2\,2}\in\,\mathfrak{sk}(n,\mathbf{C})\setminus \mathfrak{sl}(n,\mathbf{C})\).   
And, for any \(c\in\mathbf{C}\),  \([\,J(cE_{ni}),\,E_{in}\,]=\,cJ( E_{ii}+E_{nn})\).
\end{proof} 
 
\subsection{Quaternion  Lie algebra \(\widehat{ \mathfrak{sl}(n,\mathbf{H})}\) and its central extensions}

In the above we saw that \( \mathfrak{sl}(n,\mathbf{H})\) is generated as a Lie algebra by \(\mathbf{H}\otimes_{\mathbf{C}} \mathfrak{sl}(n,\mathbf{C})\).    
So that \(S^3 \mathfrak{sl}(n,\mathbf{H})\) is generated as a Lie algebra by \(S^3\mathbf{H}\otimes_{\mathbf{C}} \mathfrak{sl}(n,\mathbf{C})\).  \(S^3 \mathfrak{sl}(n,\mathbf{H})\) is a quaternion Lie algebra.
\begin{definition}~~~
The quaternion Lie subalgebra of \(\widehat{\mathfrak{gl}(n,\mathbf{H})}\) generated by \(\mathbf{C}[\phi^{\pm}]\otimes_{\mathbf{C}} \mathfrak{sl}(n,\mathbf{C})\) is called \(\mathfrak{sl}(n,\mathbf{H})\)-{\it current algebra } and is denoted by \(\widehat{\mathfrak{sl}(n,\mathbf{H})}\). 
\end{definition}

By the 2-cocycle \(\tilde{c}_k\), \(k=0,1,2\), of (\ref{trcocycle}) we have the central extension of \(\,\widehat{\mathfrak{sl}(n,\mathbf{H})}\):
 \begin{equation}
 \,\widehat{\mathfrak{sl}(n,\mathbf{H})}(a)=\widehat{\mathfrak{sl}(n,\mathbf{H})}\oplus
  (\oplus_k \mathbf{C}a_k)\,.\
 \end{equation}
 Further  we have the extension of \(\, \widehat{\mathfrak{sl}(n,\mathbf{H})}(a)\) by the derivation \(\mathbf{n}_0\) of  (\ref{radialder}): 
\begin{equation}
\widehat{\mathfrak{sl}}\,=\, \widehat{\mathfrak{sl}(n,\mathbf{H})}\oplus 
  (\oplus_k \mathbf{C}a_k)\oplus (\mathbf{C}\mathbf{n})\,.
\end{equation}
  \(\widehat{\mathfrak{sl}}\) is a quaternion Lie subalgebra of  \(\,\widehat{\mathfrak{gl}}=\widehat{\mathfrak{gl}(n,\mathbf{H})} \oplus   (\oplus_k \mathbf{C}a_k)\oplus  ( \mathbf{C}\mathbf{n})\,\). 

As a Lie subalgebras of  \(\,\widehat{\mathfrak{gl}}\,\), the Lie bracket of \(\widehat{\mathfrak{sl}}\) is given as follows:
 \begin{eqnarray*}
  [\,\phi \otimes X\, , \,\psi  \otimes Y\,]_{\widehat{\mathfrak{sl}}} &=& (\phi\,\psi)\otimes (XY)-(\psi\phi)\otimes (YX) +\, (X|Y)\, \sum_k \tilde{c}_k(\phi , \psi )\,a_k\,,  
\\[0,3cm] 
 [\,a_k\,, \phi\otimes X\,]_{\widehat{\mathfrak{sl}}}  =0\,, &\quad &\, [\,a_k,\,\mathbf{n}\,]_{\widehat{\mathfrak{sl}}} =0\,, \\[0.2cm]
  [\,\mathbf{n}\,, \phi \otimes X\,]_{\widehat{\mathfrak{sl}}}  &=&\mathbf{n}_0 \phi \otimes X\, ,
  \end{eqnarray*}
 for any bases  \(\,X,Y\in \mathfrak{sl}(n,\mathbf{C})\) and \(\phi,\,\psi\in \mathbf{C}[\phi^{\pm}]\).     
 Since 
\(
\phi ^{ +(0,0,1)} = \begin{pmatrix} 1 \\ 0 \end{pmatrix}
\) we identify a \(X\in  \mathfrak{sl}(n,\mathbf{C})\) with $\phi ^{+ (0,0,1)} \otimes X\in  \widehat{\mathfrak{sl}(n,\mathbf{H})}$.    Thus we look  
$ \mathfrak{sl}(n,\mathbf{C})$  as a Lie subalgebra of $\,\widehat{\mathfrak{sl}(n,\mathbf{H})}\,$: 
\begin{equation}
\left[\phi^{+(0,0,1)}\otimes X,\,\phi^{+(0,0,1)}\otimes Y\right]_{\widehat{\mathfrak{sl}(n,\mathbf{H})}} =\left[X,Y\right]_{\mathfrak{sl}(n,\mathbf{C})} \,,
\end{equation}
and we shall write \(\phi ^{+ (0,0,1)} \otimes X\) simply as \(X\).

\subsection{ Root space decomposition of \(\,\widehat{\mathfrak{sl}}\,\)}

  Let 
\begin{equation}
\widehat{\mathfrak{h}}\,=\,
 (\,(\, \mathbf{C}\,\phi ^{+(0, 0,1)}\,) \otimes_{\mathbf{C}}\mathfrak{h} )\,\oplus   (\oplus_k \mathbf{C}a_k)\oplus (\mathbf{C}\,\mathbf{n} )\,
 =
 \mathfrak{h}\oplus   (\oplus_k \mathbf{C}a_k)\oplus (\mathbf{C} \mathbf{n})\,,
\end{equation}
since \(\phi ^{+(0, 0,1)}=\left(\begin{array}{c}1\\0\end{array}\right)\).   
 We write \(\hat h=h+\sum s_ka_k+t \mathbf{n}\in \widehat{\mathfrak{h}}\) with \(h\in \mathfrak{h}\) and \(s_k,\,t\in\mathbf{C}\).     
 Then the adjoint actions of \(\hat h= h+\sum s_ka_k+t \mathbf{n} \in\widehat{\mathfrak{h}}\,\) on \(\,\xi=\phi\otimes X+\sum p_ja_j+q \mathbf{n}\in\widehat{\mathfrak{sl}}\) is written as follows:  
\begin{equation*}
ad(\hat h)(\xi)=\,ad(\hat h)\,(\phi\otimes X+\sum p_ja_j+q \mathbf{n})=\phi\otimes ( hX-Xh) +  t 
\mathbf{n}_0\phi\otimes X\,.
\end{equation*}

An element $\lambda$ of the dual space $\mathfrak{h}^* $ of $\mathfrak{h} $ is regarded as a  element
of $\,\widehat{\mathfrak{h}}^{\,\ast}$ by putting
\begin{align}
\left\langle \,\lambda , a_k\, \right\rangle= 
\left\langle \,\lambda , \mathbf{n} \,\right\rangle = 0,\quad k=0,1,2\,.
\end{align}
So  $\Delta \subset \mathfrak{h}^*$ is seen to be a subset of $\,\widehat{\mathfrak{h}}^{\,*}$.    
We define  $\nu\,,\,\Lambda_k\, \in \widehat{\mathfrak{h}}^{\,*}$, \( k=0,1,2\),  by
\begin{align}
\left\langle\nu , \alpha _i ^{\vee} \,\right\rangle &= \left\langle\,\Lambda_k , \alpha _i ^{\vee} \,\right\rangle = 0,   \nonumber\\[0.2cm] 
 \left\langle\,\nu , a_k\,\right\rangle &= 0\,,  \qquad \left\langle\,\nu , \mathbf{n}\,\right\rangle = 1\,,\\[0.2cm]
 \left\langle\,\Lambda_k , a_k\,\right\rangle &= 1\,,  \qquad \left\langle\,\Lambda_k , \mathbf{n}\,\right\rangle = 0\,, \quad
 1 \leqq i \leqq  r,\, \,k=0,1,2\,.\nonumber
\end{align}  
Then  the set  \(\{\,\alpha _1,\cdots,\alpha_r , \,\Lambda_0,\Lambda_1,\Lambda_2 ,\,\nu \,\} \)  forms a basis of  $\widehat{\mathfrak{h}}^{\,*}$.    

  Since \(\widehat{\mathfrak{h}}\) is a commutative subalgebra of \(\widehat{\mathfrak{sl}}\,\), 
 \(\,\widehat{\mathfrak{sl}}\) is decomposed into a direct sum of the simultaneous eigenspaces of \(ad\,(\hat h)\), \(\,\hat h\in \widehat{\mathfrak{h}}\,\).      
 
 For \(\lambda=\alpha+k_0\nu\in \widehat{\mathfrak{h}} ^{\,\ast}\), \(\alpha=\sum_{i=1}^r\,k_i\alpha_i\in\Delta\),  \(k_i\in \mathbf{Z},\,i=0,1,\cdots,r\), we put,
\begin{equation}
\widehat{\mathfrak{g}}_{\lambda}=\left\{\xi\in \widehat{\mathfrak{sl}}\,;\quad \, [\,\hat h,\,\xi\,]_{\widehat{\mathfrak{sl}}}\,=\,\langle \lambda, \hat h\rangle\,\xi\quad\mbox{ for }\, \forall\hat h\in \widehat{\mathfrak{h}}\,\right\}.\end{equation}
\(\lambda\)  is a root of \(\,\widehat{\mathfrak{sl}}\,\)  if \(\, \widehat{\mathfrak{g}}_{\lambda}\neq 0\).     \(\,\widehat{\mathfrak{g}}_{\lambda}\) is called the root space of  \(\lambda\,\).   

Let  \(\widehat{\Delta}\) denote the set of roots of  the representation \(\left(\widehat{\mathfrak{sl}}\,,ad(\widehat{\mathfrak{h}})\right)\).  

\begin{theorem}   
\begin{enumerate}
\item
\begin{eqnarray}
\widehat{\Delta} &=& \left\{ \frac{m}{2} \,\nu + \alpha\,;\quad \alpha \in \Delta\,,\,m\in\mathbf{Z}\,\right\}\nonumber \\[0.2cm]
&& \bigcup \left\{ \frac{m}{2} \,\nu\, ;\quad  m\in \mathbf{Z} \, \right\}  \,.
\end{eqnarray}
\item
For \(\alpha\in \Delta\), \(\alpha\neq 0\) and \(m\in \mathbf{Z}\), we have 
 \begin{equation}
\widehat{ \mathfrak{g}}_{\frac{m}{2}\nu+ \alpha}\,=\mathbf{C}[\phi ^{\pm};\,m\,] \otimes_{\mathbf{C}} \mathfrak{g} _{ \alpha}\,.
 \end{equation}
 \item
 \begin{eqnarray}  
  \widehat{ \mathfrak{g}}_{0\nu}&=&\,\widehat{\mathfrak{h}}\,, 
\\[0.2cm]
 \widehat{ \mathfrak{g}}_{\frac{m}{2}\nu}&=&  \,\mathbf{C}[\phi^{\pm};\,m\,]  \otimes_{\mathbf{C}} \widehat{\mathfrak{h}}\,, \,, \quad\mbox{for  \(0\neq  m\in\mathbf{Z} \) . }\,
 \end{eqnarray}
  \item
 \(\widehat{ \mathfrak{sl}}\) has the following decomposition:
\begin{equation}
\widehat{ \mathfrak{sl}}\,=\, \bigoplus_{m\in\mathbf{Z} }\, \widehat{ \mathfrak{g}}_{\frac{m}{2}\nu}\,\bigoplus\,\,\bigoplus_{\alpha\in \Delta,\, m\in\mathbf{Z} }\, 
\widehat{ \mathfrak{g}}_{\frac{m}{2}\nu+\alpha}\,
\end{equation}
\end{enumerate}
\end{theorem}

\begin{proof} ~~~ First we prove the second assertion.      
Let \(X\in\mathfrak{g}_{\alpha}\) for a \(\alpha\in \Delta\), \(\alpha\neq 0\), and let  \(\varphi\in \mathbf{C}[\phi^{\pm};\,m]\)  for a \(m\in \mathbf{Z}\).     We have, for any \(h\in\mathfrak{h}\), 
\begin{eqnarray*}
[\,\phi ^{+ (0,0,1)} \otimes h , \,\varphi \otimes X\,] _{ \widehat{\mathfrak{sl}}}&=&
 \varphi \otimes (hX-Xh) \,
= \left\langle \alpha , h\right\rangle \varphi \otimes X,
\\[0.2cm]
[\,\mathbf{n}, \,\varphi \otimes X\,]_{ \widehat{\mathfrak{sl}}}&=& \frac{m}{2} \varphi\otimes X,
\end{eqnarray*}
that is, for every \(\hat{h} \in \widehat{\mathfrak{h}}\), we have 
\begin{equation*}
[\,\hat h\,, \varphi \otimes X]_{ \widehat{\mathfrak{sl}}} = \left\langle \frac{m}{2} \delta + \alpha \,, \,\hat h\,\right\rangle (\varphi\otimes X)\,.
\end{equation*}
Therefore \(\varphi\otimes X\in \widehat{\mathfrak{g}}_{\frac{m}{2}\nu+ \alpha}\). 

Conversely,  for a given \(m\in\mathbf{Z}\)  and  a \(\xi\in \widehat{\mathfrak{g}}_{\frac{m}{2}\nu+ \alpha}\), we shall show that \(\xi\)  has the form  \(\,\phi\otimes X\,\) with \(\phi\in \mathbf{C}[\phi^{\pm};m]\,\) and \(X\in \mathfrak{g}_ \alpha\,\) .   Let 
\(\xi=\phi\otimes X+\sum_{k\in \mathbf{Z}}p_k a_k+ q\mathbf{n}\,\) for \(\phi\in \mathbf{C}[\phi^{\pm}]\), \(\,X\in \mathfrak{sl}(n,\mathbf{C})\) and \(p_k,\,q\in\mathbf{C}\).     \(\phi\) is  decomposed to the sum  
\[\phi=\sum_{n\in \mathbf{Z}}\,\phi_n\]
by the homogeneous degree; \(\phi_n\in \mathbf{C}[\phi^{\pm};n]\).
 We have 
\begin{eqnarray*}
&&[\hat h,\xi]_{ \widehat{\mathfrak{sl}}} =[\,\phi^{+(0,0,1)}\otimes h+ \sum s_ka_k+t\mathbf{n}\,, \,\phi \otimes X\,+\sum\, p_k a_k+q \mathbf{n}\,]_{ \widehat{\mathfrak{sl}}} 
 =\,\phi\otimes [\,h\,,\,X\,]\\[0.2cm]
 &&\qquad + \,t (\,\sum_{n\in \mathbf{Z}} \,\frac{n}{2} \phi_n  \,\otimes X\,)
\end{eqnarray*}
for any \(\hat h=\phi^{+(0,0,1)}\otimes h+\sum s_ka_k+t\mathbf{n}\in \widehat{\mathfrak{h}}\,\).
From the assumption we have 
\begin{eqnarray*}
 [\,\hat h,\xi\,]_{ \widehat{\mathfrak{sl}}} \,&=&\,\langle\, \frac{m}{2}\delta+ \alpha\,,\,\hat h\,\rangle\, \xi\,\\[0.2cm]
&=& < \alpha,\,h>\phi\otimes X\, +(\frac{m}{2}t+< \alpha,\,h>)(\sum p_k a_k+q\mathbf{n})\,\\[0.2cm]
&& \quad+\,
\frac{m}{2}t\,
  (\sum_{n} \,\phi_{n} )\otimes X .
\end{eqnarray*}
Comparing the above two equations we have \(p_k=q=0\), and $\phi_n = 0$ for all \(n\) 
except for $n = m$.      Therefore    $ \phi \in\mathbf{C}[\phi^{\pm};m]$.   We also have  $[\hat h , \xi]_{ \widehat{\mathfrak{sl}}}  =\phi\otimes [h,X]= \langle \alpha ,\,h\rangle \,\phi \otimes X$ for all \(\hat h=\phi^{+(0,0,1)}\otimes h+\sum s_ka_k+td \in\widehat{\mathfrak{h}}\).     Hence  $X \in \mathfrak{g}_{ \alpha}$ and 
  \(\xi= \phi_m\otimes X \in \widehat{\mathfrak{g}}_{\frac{m}{2}\nu+ \alpha}\,\).   We have proved
  \begin{equation*}
  \widehat{\mathfrak{g}}_{\frac{m}{2}\nu+\alpha}=\mathbf{C}[\phi^{\pm};m] \otimes_{\mathbf{C}} \mathfrak{g}_{ \alpha}\,.
 \end{equation*}
 The proof of the third assertion is also carried out by the same argument as above if  we revise it for the case \(\alpha=0\) .      The above discussion yields the first and the fourth assertions.
  \end{proof}

\begin{proposition}~~~
We have the following relations: 
\begin{enumerate}
\item
\begin{equation}
\left[\, \widehat{\mathfrak{g}}_{\frac{m}{2}\nu+\alpha}\,,\, \widehat{\mathfrak{g}}_{\frac{n}{2}\nu+\beta}\,\right ]_{ \widehat{\mathfrak{sl}}}\,\subset \,
 \widehat{\mathfrak{g}}_{\frac{m+n}{2}\nu+\alpha+\beta}\,\,,
 \end{equation}
 for \(\alpha,\,\beta \in \widehat{\Delta}\) and for  \(m,n\in\mathbf{Z}\).
\item
\begin{equation}
\left[\, \widehat{\mathfrak{g}}_{\frac{m}{2}\nu}\,,\, \widehat{\mathfrak{g}}_{\frac{n}{2}\nu}\,\right ]_{\widehat{\mathfrak{sl}}}\,\subset \,
 \widehat{\mathfrak{g}}_{\frac{m+n}{2}\nu}\,, 
 \end{equation}
  for  \(m,n\in\mathbf{Z}\).
\end{enumerate}
\end{proposition}

The assertions are proved by a standard argument using the properties of Lie bracket.

 \subsection{ Chevalley generators of \(\widehat{\mathfrak{sl}}\)}

By the natural embedding of \(\mathfrak{sl}(n,\mathbf{C})\) in \(\widehat{\mathfrak{sl}}\) we have the vectors 
\begin{eqnarray}
\widehat h_i&=&\phi^{+(0,0,1)}\otimes h_i\,\in \widehat{\mathfrak{h}},\,\nonumber\\[0.2cm]
\widehat x_i&=&\phi^{+(0,0,1)}\otimes x_i\,\in \widehat{\mathfrak{g}}_{0\nu+\alpha_i},\quad \widehat y_i=\phi^{+(0,0,1)}\otimes y_i\,\in \widehat{\mathfrak{g}}_{0\nu-\alpha_i},\qquad i=1,\cdots,r=n-1\,.
\nonumber\end{eqnarray}
Then 
\begin{eqnarray}
\left[\widehat x_i\,,\widehat y_j\,\right]_{ \widehat{\mathfrak{sl}}} &=&\,\delta_{ij}\,\widehat h_i\,,\nonumber\\[0.2cm]
\left[\widehat h_i\,,\widehat x_j\,\right ]_{ \widehat{\mathfrak{sl}}}&=&\,a_{ij}\,\widehat x_j,\quad 
\left[\widehat h_i\,,\widehat y_j\,\right]_{ \widehat{\mathfrak{sl}}} =\,- a_{ij}\,\widehat y_j\,,\quad 1\leq i,j\leq r .
\end{eqnarray}
We have obtained a part of generators of \(\widehat{\mathfrak{sl}}\) that come naturally from \(\mathfrak{sl}(n,\mathbf{C})\).   We want to augment these generators to the Chevalley generators of  \(\widehat{\mathfrak{sl}}\).    
We take the following set of generators of  the algebra \(\mathbf{C}[\phi^{\pm}]\vert\,S^3\):
\begin{eqnarray}
I&=\phi^{ +(0,0,1)}=\left(\begin{array}{c}1\\0\end{array}\right),\,\qquad 
J&=\phi^{ +(0,0,0)}=\left(\begin{array}{c}0\\-1\end{array}\right)\,,\nonumber
\\[0.2cm]
  \kappa&=\phi^{+(1,0,1)}\,=\,\left(\begin{array}{c} z_2\\-\overline{z}_1\end{array}\right),\qquad
 \lambda&= \,\phi^{ -(0,0,0)}=\left(\begin{array}{c} z_2\\ \overline{z}_1\end{array}\right).
\end{eqnarray}

  We put
\begin{eqnarray*}
\kappa_{\ast}&=&\,\frac{-\sqrt{-1}}{\sqrt{2}}\phi^{+(1,1,2)}+\frac{\sqrt{-1}}{2}(\phi^{-(0,0,0)}-\phi^{+(1,0,1)})
\,=\,\sqrt{-1}\left(\begin{array}{c}  \overline z_2 \\ \overline z_1\end{array}\right)\, \\[0.2cm]
\lambda_{\ast}&=&\,\frac{-\sqrt{-1}}{\sqrt{2}}\phi^{+(1,1,2)}-\,\frac{\sqrt{-1}}{2}(\phi^{-(0,0,0)}-\phi^{+(1,0,1)})
\,=\,\sqrt{-1}\left(\begin{array}{c}  \overline z_2 \\ -\overline z_1\end{array}\right)\, 
\end{eqnarray*}

\begin{lemma}~~
\begin{enumerate}
\item
\begin{equation}\kappa\,\in \mathbf{C}[\phi^{\pm};1]\,,\qquad\,\lambda\,\in \mathbf{C}[\phi^{\pm};-3]\,.
\end{equation}
\item
\begin{eqnarray}
\,c_0(\kappa,\kappa_{\ast})\,&=-1\,,\quad c_1(\kappa,\kappa_{\ast})=c_2(\kappa,\kappa_{\ast})=0,\\[0.2cm]
c_0(\lambda,\lambda_{\ast})&=-1\,,\quad \,c_1(\lambda,\lambda_{\ast})=c_2(\lambda,\lambda_{\ast})=0\,.
\end{eqnarray}
\end{enumerate}
\end{lemma}

Let \(\theta=\alpha_1+\cdots+\alpha_{l-1}\) be the highest root of \(\mathfrak{sl}(n,\mathbf{C})\) and suppose that \(x_\theta\in\mathfrak{g}_{\theta}\) and \(y_\theta\in\mathfrak{g}_{-\theta}\) satisfy the relations \([x_\theta\,,\,y_{\theta}]\,=\,h_\theta\) and \((x_\theta\vert y_\theta)=1\).    We introduce the following vectors of \(\widehat{\mathfrak{sl}}\);
\begin{align}
 \widehat y_{J}&=J\otimes y_{\theta}\,\in \widehat{\mathfrak{g}}_{0\nu-\theta}\,,\quad &
  \widehat x_{J}&=(-J)\otimes x_{\theta} \,\in \widehat{\mathfrak{g}}_{0\nu+\theta}\,,\\[0.2cm]
   \widehat y_{\kappa}&=\kappa\otimes y_{\theta}\,\in \widehat{\mathfrak{g}}_{\frac12\nu-\theta}\,,\quad &
  \widehat x_{\kappa}&=\kappa_{\ast} \otimes x_{\theta} \,\in \widehat{\mathfrak{g}}_{-\frac32\nu+\theta}\oplus \widehat{\mathfrak{g}}_{\frac12\nu+\theta}\,,\\[0.2cm]
  \widehat y_{\lambda}&=\lambda \otimes y_{\theta}\,\in \widehat{\mathfrak{g}}_{-\frac32\nu-\theta}\,,\quad &
  \widehat x_{\lambda}&=\lambda_{\ast}\otimes x_{\theta} \,\in \widehat{\mathfrak{g}}_{-\frac{3}{2}\nu+\theta}\oplus \widehat{\mathfrak{g}}_{\frac12\nu+\theta}\,.
\end{align}
Then we have the generators of \(\widehat{\mathfrak{sl}(n,\mathbf{H})}(a)\)  that are given by the following three triples:
\begin{eqnarray}
&& \left(\,\widehat{x}_i,\widehat{y}_i,\widehat{h}_i \right) \quad  i=1,2,\cdots,r,\nonumber\\[0.2cm]
&&\left( \widehat{x}_{\lambda}, \widehat{y}_{\lambda},\widehat h_{\theta}\right),\quad 
\left( \widehat{x}_{\kappa}, \widehat{y}_{\kappa},\widehat h_{\theta}\,\right),\quad
 \,\left( \widehat{x}_{J}, \widehat{y}_{J},\widehat h_{\theta}\right)\, \,.
\end{eqnarray}
These three triples satisfy the following relations.
\begin{proposition}~~~
\begin{enumerate}
\item
\begin{equation}
\left[\,\widehat x_{\pi}\,,\,\widehat y_i\,\right]_{\widehat{\mathfrak{sl}}}=\,
\left[\,\widehat y_{\pi}\,,\,\widehat x_i\,\right]_{\widehat{\mathfrak{sl}}} =0\,,\quad\mbox{for } \,1\leq i\leq r ,\,\mbox{ and }\, \pi=J,\,\kappa,\,\lambda\,.\label{i}
\end{equation}
\item
\begin{equation}
\left[\,\widehat x_{J}\,,\,\widehat y_J\,\right]_{\widehat{\mathfrak{sl}}}=\,\widehat  h_\theta\,,
\end{equation}
\item
\begin{equation}\quad
\left[\,\widehat x_{\lambda}\,,\,\widehat y_{\lambda}\,\right]_{\widehat{\mathfrak{sl}}}=\sqrt{-1}\,\widehat  h_\theta-a_0 ,\quad
\left[\,\widehat x_{\kappa}\,,\,\widehat y_{\kappa}\,\right]_{\widehat{\mathfrak{sl}}}=\sqrt{-1}\,\widehat  h_\theta\,-a_0\,.
\end{equation}
\end{enumerate}
\end{proposition}
Adding \(\mathbf{n}\) to these generators of \(\widehat{\mathfrak{sl}(n,\mathbf{H})}(a)\) we obtain the Chevalley generators of 
\(\widehat{\mathfrak{sl}}\).

\end{document}